%% file: ComplexityOfKnotoidArxiv.tex
\documentclass[12pt,twoside]{article}
\pdfoutput=1
\usepackage[cp1251]{inputenc}   
\usepackage[english]{babel}

\usepackage{mathtools}
\usepackage{amsmath}
\usepackage{amsfonts,amssymb}
\usepackage{mathtools}

\usepackage{amsmath,amssymb,amsopn,amsthm}
\usepackage{graphicx}
\usepackage{epsfig, epsf}
\usepackage{epstopdf}
\usepackage{subfigure}
\usepackage{tikz}
\usetikzlibrary{knots}
\usetikzlibrary{decorations.pathmorphing,decorations.markings}
\usetikzlibrary{arrows,positioning}
\usetikzlibrary{svg.path}

\theoremstyle{plain}
\newtheorem{lemma}{Lemma}
\newtheorem{theorem}{Theorem}

\newtheorem{corollary}{Corollary}

\newtheorem{remark}{Remark}

\newcommand{\crn}{\operatorname{cr}}
\newcommand{\spn}{\operatorname{spn}}
\newcommand{\Int}{{\mbox{Int\ }}}
\newcommand{\hh}{\operatorname{h}}

\newcommand{\fkd}{FKD}

\def\figuresPointSize{0.2}  
\def\figuresSmallPointSize{0.15}  
\def\figuresLineThickness{1.0}  
\def\figuresThinLineThickness{0.75}  
\def\figuresThickLineThickness{1.75}  
\begin{document}



\title{A relation between the crossing number and the height of a knotoid}

\author{
Philipp Korablev\footnote{Chelyabinsk State University, Chelyabinsk, Russia,
Krasovskii Institute of Mathematics and Mechanics, Ural Branch of the Russian Academy of Sciences, Yekaterinburg, Russia
korablev@csu.ru} \and Vladimir Tarkaev\footnote{Chelyabinsk State University, Chelyabinsk, Russia,
Krasovskii Institute of Mathematics and Mechanics, Ural Branch of the Russian Academy of Sciences, Yekaterinburg, Russia,
St. Petersburg State University, Saint Petersburg, Russia
trk@csu.ru}
}

\maketitle

\begin{abstract}
Knotoids are open ended knot diagrams regarded up to 
Reidemeister moves and isotopies.
The notion is introduced by V.~Turaev in 2012.
Two most important numeric characteristics of a knotoid are the crossing number and the height.
The latter is the least number of intersections between
a diagram and an arc connecting its endpoints,
where the minimum is taken over all representative diagrams and all 
such an arcs disjoint from crossings.
In the paper we answer the question: are there any relations between the crossing number and the height of a knotoid.
We prove that
the crossing number of a knotoid is greater than or equal to twice the height of the knotoid.
Combining the inequality with known lower bounds of the height
we obtain a lower bounds of the crossing number of a knotoid
via the extended bracket polynomial, the affine index polynomial
and the arrow polynomial of the knotoid.
As an application of our result we prove an upper bound for the length of a bridge
in a minimal diagram of a classical knot:
the number of crossings in a minimal  diagram of a knot is greater than or equal to
three times the length of a longest bridge in the diagram.
\end{abstract}



\section{Introduction}
\label{sec:Introduction}

The concept of knotoid is introduced by V.~Turaev~\cite{TuraevKnotoids}.
Then the subject was investigated by a few groups of researchers,
mainly by L.~Kauffman and his collaborators.
For a survey of existing works in the area including an application to biology
see~\cite{GKL}.
 For comprehensive tables of knotoids
see~\cite{Bartholomew}, \cite{GDS} and~\cite{KorablevMayTarkaev}.

Intuitively, knotoids can be considered as open-ended knot-type pictures up to an appropriate equivalence.
More precisely, knotoid diagrams are generic immersions of the unit interval into a surface,
together with the under/over-crossing information at double points.
Knotoids are defined as the equivalence classes of knotoid diagrams under isotopies and the Reidemeister moves
(precise definitions are given in Section~\ref{sec:Preliminaries}).
In~\cite{TuraevKnotoids} Turaev shows
that knotoids in $S^2$ generalize knots in $S^3$
and that knotoids are closely related to knots in thickened surfaces via the closure operation
(about injectivity and surjectivity of the closure map in the case of knotoids in $S^2$ see~\cite{KorablevMay}).
Later  in~\cite{KauffmanInvariants}
Kauffman and Gugumcu introduced and studied virtual knotoids
which generalize classical knotoids likewise virtual knots generalize classical knots.

One of the most important characteristic of a knotoid is the crossing number
which is a direct analogue of that for knots.
The problem of determining  the exact value of the crossing number of a knotoid is very complicated.
Any diagram gives an upper bound of the value but not many lower bounds are known.
We mention proved by Turaev in~\cite{TuraevKnotoids}
a generalization to knotoids Kauffman's inequality
relating the span of the bracket polynomial to the crossing number.

One more important numeric characteristic of a knotoid (in $S^2$ only) is the height.
The notion has no direct analogue in classical theory.
It is introduced by Turaev in~\cite{TuraevKnotoids}
under the name of ``the complexity of a knotoid''.
To define the value
consider an arc connecting the endpoints of a diagram of a knotoid.
In general the arc intersects with the diagram.
The height is the minimum of the number of the intersections over 
all representative diagrams
and all such an arcs disjoint from crossings
(see Section~\ref{sec:Preliminaries} for precise definition).
Turaev in~\cite{TuraevKnotoids}
obtained a lower bound for the height of a knotoid via the extended bracket polynomial
which is a purely knotoid generalization of the Kauffman bracket polynomial
(see also~\cite{Kutluay}
where corresponding Khovanov-type invariant is constructed).
In~\cite{KauffmanInvariants}
some known polynomial invariants of virtual knots
are extended to the case of classical and virtual knotoids,
and, in particular, it is shown that these invariants give a lower bounds for the height of a knotoid.

The main goal of our paper is to relate the crossing number to the height of the same knotoid.
Both these values are minima over all representative diagrams
but in general they can be reached at two different representatives.
Theorem~\ref{theorem:MainResult} 
(Section~\ref{sec:MainResult})
states that $\crn(K) \geq 2 \hh(K)$ where $\crn(K)$ and $\hh(K)$
denote the crossing number and the height of a knotoid $K$, respectively.
Combining the inequality with mentioned above lower bounds of the height of a knotoid
we obtain a lower bounds of its crossing number
via the extended bracket polynomial
(Section~\ref{sec:ExtendedBracketPolynomial}),
and via the affine index polynomial and the arrow polynomial
(Section~\ref{sec:ViaKauffman}).

One more application of our result is a necessary condition
for a diagram of a classical knot  to be a minimal in the sense of the number of crossings
(Section~\ref{sec:Bridge}):
if a diagram $D$ of a classical knot is minimal then
$\crn(D) \geq 3k$ where $\crn(D)$ and $k$ denote respectively the number of crossing in the diagram $D$
and the length of a longest bridge in $D$.
The statement shows that approaches coming from the knotoid theory
can be useful to the theory of classical knots.
As another such example we can mention~\cite{SeifertGenus}
where the authors use introduced in~\cite{TuraevKnotoids}
 the notion of Seifert surface of a knotoid
to define an procedure
which under some conditions gives a better estimate for the Seifert genus of a knot
than the one obtaining directly by given diagram.

The paper is organized as follows.
Section~\ref{sec:Preliminaries}
gives main definitions and using notation.
In Section~\ref{sec:MainResult}
we formulate the main result and three its corollaries.
Section~\ref{sec:ProofOfMainResult}
contains a proof of the main result
which is divided  into several auxiliary statements.
The last Section~\ref{sec:ProofOfCorollary}
gives a proof of Corollary~\ref{corol:Bridge}
formulated in Section~\ref{sec:Bridge}.
Two other Corollaries~\ref{corol:spn_u}
and~\ref{corol:Affine}
do not need in   a proof because they are
the direct consequences of our Theorem~\ref{theorem:MainResult}
and theorems formulated
in Section~\ref{sec:ExtendedBracketPolynomial}
and~\ref{sec:ViaKauffman}, respectively.

\section{Preliminaries}
\label{sec:Preliminaries}

A \emph{knotoid diagram} $D$ in $2$-sphere $S^2$
  is a generic immersion of the (closed) segment $[0,1]$ into $S^2$
whose only singularities are transversal
double points endowed with standard over/under-crossing data.
The images of $0$ and~$1$ under
this    immersion are called the    \emph{beginning} and the    \emph{end} of $D$, respectively.
   These two points are   distinct from each  other and from the double points; they are 
  called the \emph{endpoints} of~$D$.
The double points  of $D$ are    called the \emph{crossings} of $D$.
  
Turaev in~\cite{TuraevKnotoids} 
considers knotoid diagrams in $\mathbb{R}^2$
and in an orientable surface of arbitrary  genus with (maybe) non-empty boundary.
In the paper we restrict ourselves with knotoid diagrams in $S^2$ only
and throughout saying about knotoid diagrams we mean knotoid diagrams in $S^2$.

Knotoid diagrams $D_1$ and $D_2$ are \emph{(ambient) isotopic}
if there is an isotopy of $S^2$ in itself transforming $D_1$ in $D_2$.
In particular, an isotopy of a knotoid diagram   may displace the endpoints.

\input{Reidemeister}

  We define  three  \emph{Reidemeister moves} $\Omega_1, \Omega_2, \Omega_3$ on   knotoid
diagrams.  The move $\Omega_i$ 
(see Fig.~\ref{Figure:Reidemeister})
on a    knotoid diagram $D$ preserves~$D$ outside a closed $2$-disk disjoint
   from the endpoints and modifies $D$ within this
   disk   as the standard  $i$-th Reidemeister move, for $i=1,2,3$  (pushing a branch of
$D$ over/under the endpoints is not
   allowed).

A \emph{knotoid} is defined to be an equivalence class of knotoid diagrams
under the Reidemeister moves $\Omega_1,\Omega_2,\Omega_3$ and ambient isotopies.

The \emph{crossing number} of a knotoid $K$ is
the minimal number of crossings over all representative diagrams;
we denote the value by $\crn(K)$.

Given a knotoid diagram $D$, a \emph{shortcut} of $D$ is
an embedded oriented arc $\gamma$
  starting at the beginning  of $D$,
ending at the end of $D$  and otherwise meeting~$D$ transversely
at a finite set of points distinct from the crossings of $D$.

The \emph{height of a knotoid diagram} $D$ is defined to be the minimum  over all shortcuts of $D$
of the number of points in which the interior of a shortcut intersects with $D$.
The \emph{height of a knotoid} is the minimum of height
over all representative diagrams.
We denote by $\hh(D)$ and $\hh(K)$ the height of a knotoid diagram $D$ and of a knotoid $K$, respectively.

The notion of the height of knotoid was introduced by Turaev
in~\cite{TuraevKnotoids}
under the name ``complexity of a knotoid''.
We prefer the term ``height'' which was proposed in~\cite{KauffmanInvariants}.

\section{The main result}
\label{sec:MainResult}

\begin{theorem}
\label{theorem:MainResult}
For a knotoid $K$
\begin{equation} \label{eq:Main}
\crn(K) \geq 2 \hh(K)
\end{equation}
and there exists an infinite family of knotoids
for which the inequality~\eqref{eq:Main}
becomes equality.
\end{theorem}

The second part of the theorem proves by the infinite family
of spiral knotoids considered in~\cite[Section 4]{KauffmanInvariants}.

The first part of theorem will be proved below in Section~\ref{sec:ProofOfMainResult}.
Before the proof we give three applications of the result.

\subsection{A lower bound for the crossing number of a knotoid via the extended bracket polynomial}
\label{sec:ExtendedBracketPolynomial}

V.~Turaev  in~\cite{TuraevKnotoids} introduced the extended bracket polynomial of a knotoid.
Note that L.~Kauffman~\cite{KauffmanExtended} 
has used the term ``extended bracket polynomial''
for another polynomial which is an invariant of virtual knots and links.

Turaev's extended bracket polynomial is a Laurent polynomial
$\langle \langle K \rangle \rangle_{\circ}(a,u) \in \mathbb{Z}[a^{\pm 1},u^{\pm 1}]$.
Here the variable $a$ has the same sense as in the case
of classical Kauffman bracket polynomial
while the variable $u$ counts intersections of curves in each state with a shortcut
(for details see~\cite[Section 8]{TuraevKnotoids}).
In particular, among other properties of the polynomial Turaev establishes following inequality.

\cite[Section 8.3]{TuraevKnotoids} 
For a knotoid $K$
$$\spn_u (\langle \langle K \rangle \rangle_{\circ}) \leq 2 \hh(K)$$
where $\spn_u( \, )$ denotes the span of the polynomial with respect to  the variable $u$
(i.e., the difference between the maximal and the minimal degrees of  variable $u$
involved in the polynomial).

Combining the inequality with~\eqref{eq:Main}
we obtain following statement.

\begin{corollary} \label{corol:spn_u}
For a knotoid $K$
$$\crn(K) \geq \spn_u(\langle \langle K \rangle \rangle_{\circ}).$$
\end{corollary}

\subsection{A lower bound for the crossing number of a knotoid via the affine index polynomial and the arrow polynomial}
\label{sec:ViaKauffman}

The affine index polynomial~\cite{KauffmanAffine}
and the arrow polynomial~\cite{KauffmanArrow}
are known invariants of virtual knots and links.
In~\cite{KauffmanInvariants}
these invariants are generalized to the
classical (i.e., in $S^2$) and virtual knotoids.
In particular, the authors establish following lower bound estimations for the height of a knotoid.

\cite[Theorem 4.12]{KauffmanInvariants}
Let $K$ be a classical knotoid.
The height of $K$ is greater than or equal to the maximum degree of the affine index
polynomial of $K$.

\cite[Theorem 5.4]{KauffmanInvariants}
The height of a classical knotoid $K$ is greater than or equal to the $\Lambda$-degree of
its arrow polynomial.

Combining these theorems with Theorem~\ref{theorem:MainResult}
we obtain following statement.

\begin{corollary} \label{corol:Affine}
For a knotoid $K$
$$
\crn(K) \geq 2 d_{max}(K) \quad \text{and} \quad \crn(K) \geq 2 d_{\Lambda}(K)
$$
where $d_{max}(K)$ and $d_{\Lambda}(K)$ denote respectively
the maximum degree of the affine index polynomial 
and $\Lambda$-degree of the arrow polynomial
of $K$.
\end{corollary}

\subsection{Minimality of a knot diagram and the length of the longest bridge}
\label{sec:Bridge}

Consider a diagram $D$ of a classical knot.
Recall that
\emph{an over-bridge} (resp. \emph{under-bridge}) 
of length $k$ of a diagram $D$ is a 
consecutive sequence of $k$ over-crossings (resp.
under-crossings) in $D$.

\begin{corollary} \label{corol:Bridge}
If a diagram $D$ of a knot is minimal with respect to the number of crossings
and $k(D)$ is the maximum of the length over all bridges 
(both over-bridges and under-bridges) in the diagram $D$
then
$$\crn(D) \geq 3 k(D)$$
where $\crn(D)$ denotes the number of crossings in $D$.
\end{corollary}
 
We prove Corollary~\ref{corol:Bridge}
in~Section~\ref{sec:ProofOfCorollary}
just after a proof of~Theorem~\ref{theorem:MainResult}.

\section{Proof of Theorem~\ref{theorem:MainResult}}
\label{sec:ProofOfMainResult}

The proof is divided into following  $4$ parts.
Firstly in Sections
~\ref{sec:FlatKnotoidDiagram},
\ref{sec:Prime},
we give a few additional definitions,
in particular, the definition of a flat knotoid diagram
and the definition of a prime flat knotoid diagram.
Then in Sections~
\ref{sec:GammaRegions}
--\ref{sec:ProofForPrime}
we prove the inequality $\crn(F) \geq 2 \hh(F)$ for a prime flat knotoid diagram $F$.
Then in Section~
\ref{sec:ProofForArbitrary}
we prove the inequality for any (including non-prime) flat knotoid diagrams.
Finally in Section~
\ref{sec:ProofForKnotoid}
we consider an arbitrary knotoid,
and this completes the proof of Theorem~\ref{theorem:MainResult}.

\subsection{Flat knotoid diagram}
\label{sec:FlatKnotoidDiagram}

A \emph{flat knotoid diagram} (or \emph{\fkd}~for short)
is a knotoid diagram of which crossings are not equipped with over/under-information.
Therefore, all crossings of \fkd~(if any) are flat crossings
consisting in the transversal intersections of strands without any over/under-crossing information.
A \fkd~is called \emph{trivial} if it has no crossings
(i.e., if it is an embedding of $[0,1]$).
The notions  of endpoints, shortcut and height of a \fkd~are defined
in analogy with the notions of endpoints, shortcut and height of knotoid diagram.

It is clear that a \fkd~can be viewed as a graph embedded into $S^2$
satisfying some conditions coming from its relation with a knotoid diagram.
In particular, a \fkd~has two univalent vertices (the endpoints)
and all its other vertices (if any) are $4$-valent,
the latter vertices are called the \emph{crossings} of the \fkd.
The edges of the graph is called the \emph{edges} of the \fkd.
Two edges adjacent to the endpoints are called the \emph{outer} edges.
Given a \fkd~$F$,
connected components of the set $S^2 \setminus F$
are called the \emph{regions} of the \fkd~$F$.

The notion of \fkd~is introduced
in~\cite[Section 3.2]{KauffmanInvariants}.
Recently  Turaev in~\cite{TuraevVirtualStrings}
has studied an ``open strings'' 
which also can be regarded as a generic immersions of a segment into a surface.
But it is necessary to emphasize that
in both these works the authors  take an interest in an equivalence classes of corresponding objects
while we below deal with an individual \fkd.
Throughout a \fkd~(unlike knotoid diagrams) is regarded as immovable,
the only what can be changed is its shortcut.

\subsection{Prime flat knotoid diagram}
\label{sec:Prime}

A \fkd~$F$ is called \emph{prime} if following two conditions hold:
\begin{itemize}
\item[(i)]
Every  embedded circle meeting $F$ transversely in exactly two
points bounds a disk meeting $F$ along a proper embedded
arc or along two disjoint   embedded arcs adjacent to the endpoints of $F$.
\item [(ii)]Every embedded circle meeting $F$ transversely in exactly one point
bounds a regular neighborhood of one of the endpoints of $F$.
\end{itemize}

\input{NonPrime}

The definition is a direct analog of Turaev's definition of prime knotoid diagram
\cite[Section 7.3]{TuraevKnotoids}.
Two examples of non-prime \fkd~are shown 
in Fig.~\ref{Figure:NonPrime}.
The one depicted on the left-hand side does not satisfies the condition~(i),
the other one does not satisfies the condition~(ii).

\subsection{$\gamma$-edges and $\gamma$-regions}
\label{sec:GammaRegions}

Let $F$ be a \fkd~with a shortcut $\gamma$.
We need following definitions.
\begin{itemize}
\item [] \emph{Minimal shortcut}:
a shortcut $\gamma$ is called the \emph{minimal shortcut} of $F$
if $|\Int \gamma \cap F| = \hh(F)$,
where  $\Int \gamma$ denotes the interior of the shortcut $\gamma$
($|*|$ here and below denotes the cardinality of the corresponding set).

\item [] \emph{$\gamma$-edge}:
an edge $e$ of $F$ is called the \emph{$\gamma$-edge} of $F$
if either $e$ is an outer edge of $F$
or $e \cap \gamma \not= \emptyset$.

\item [] \emph{$\gamma$-region}:
a connected component $\Delta$ of the set $S^2 \setminus F$ is called the \emph{$\gamma$-region} of $F$
if $\Delta \cap \gamma \not=\emptyset$.
\end{itemize}

\input{Gamma}

In Fig.~\ref{Figure:Gamma}
$\gamma$-regions of depicted \fkd~are shaded,
its $\gamma$-edges are drawn with thick lines.

\input{gammaReduction0102}

\begin{theorem}
\label{Theorem:GammaMinimum}
Let $F$ be a non-trivial \fkd~with minimal shortcut $\gamma$.
\begin{enumerate}
\item If $\Delta$ is a $\gamma$-region
then $\gamma$  intersects $\partial \Delta$ in exactly two points
which lie inside two distinct $\gamma$-edges.
\item If $F$ is prime and both regions adjacent to an edge $e$ are $\gamma$-regions
then $e$ is the $\gamma$-edge.
\end{enumerate} 
\end{theorem}

\begin{proof}
{\bf 1}.
By definition of a shortcut  its endpoints are distinct
and lye outside a $\gamma$-region,
the interior of a shortcut intersects with a \fkd~transversely in a finite number of points.
By definition of a $\gamma$-region $\Delta$ is open connected and $\partial \Delta \subset F$.
Hence the number of intersections of $\gamma$ with $\partial \Delta$ is finite and is greater than $1$.
Let $\partial \Delta \cap \gamma =\{p_1,\ldots,p_n\},n \geq 2,$
and let the points $p_1,\ldots,p_n$ are numbered in the order
in which $\gamma$ passes through them.
Suppose $n >2$. In this case
the part of $\gamma$ between $p_1$ and $p_n$ contains at least one intersection with $F$.
This contradicts the minimality of the shortcut $\gamma$.
Indeed, we can replace the part $[p_1,p_n] \subset \gamma$ with a simple arc 
connecting $p_1$ with $p_n$ and lying inside $\Delta$
(see~Fig.~\ref{Figure:gammaReduction0102} on the left),
and resulting shortcut has less intersections with $F$ than the initial one.

Therefore, $\gamma \cap \partial \Delta =\{p_1,p_2\}$
and it remains to show that $p_1,p_2$ can not lie in the same $\gamma$-edge.
Assume, to the contrary,
that there exists an edge $e$ such that $p_1,p_2 \in e$.
If $e$ is an outer edge (say the first one) 
then $p_2$ can not be the other endpoint of the edge $e$
because, by definition,  a shortcut  does not passes through a crossings 
 and, by hypothesis, $F$ is non-trivial.
So we can  decrease the number of intersections of $\gamma$ with $F$ by $1$
connecting $p_1$ with a point lying in $\gamma$ just after $p_2$.
The case $p_1,p_2$ lie in the last edge of $F$ is similar to previous one.
If $e$ is not an outer edge then $\gamma$ crosses $e$ in different directions
(it comes into $\Delta$ in $p_1$ and then goes out in $p_2$)
since otherwise $|\gamma \cap \partial \Delta| \geq 3$.
So we can decrease the number of intersections by $2$ pushing the arc $[p_1,p_2] \subset \gamma$
from $\Delta$ across the edge $e$
(see~Fig.~\ref{Figure:gammaReduction0102} on the right).

{\bf 2}.
If the edge $e$ is an outer edge then it is an $\gamma$-edge by definition.

Let $e$ is not an outer edge. 
Suppose $e$ is not a $\gamma$-edge, i.e., $\gamma \cap e =\emptyset$.
Denote $\gamma$-regions
adjacent to $e$ by $\Delta_1, \Delta_2$.
We consider two cases depending on whether coincide these regions or not.

{\bf $\Delta_1 \not=\Delta_2$}.
The shortcut $\gamma$ is minimal, the regions $\Delta_1,\Delta_2$ are $\gamma$-regions
and, by assumption,  $\gamma \cap e =\emptyset$
hence there exists an $\gamma$-edge $e'$
such that  $\gamma$-regions adjacent to $e'$ are the same $\Delta_1,\Delta_2$.
Thus $\Delta_1$ and $\Delta_2$ have two different common edges.
In the situation there exist  an embedded circle which intersects with $F$ in exactly two points
lying inside $e$ and $e'$.
Both disks bounded by the circle contain crossings of $F$
(the endpoints of $e$ and $e'$).
The existence of such a circle contradicts to condition (i) of the definition of prime \fkd.

{\bf $\Delta_1 =\Delta_2$}.
In particular, it means that the endpoints of $e$ do not coincide.
In this situation there exists an embedded circle which
intersects with $F$ in exactly one point lying in the edge $e$.
Both disks bounded by the circle contain crossings of $F$
(the endpoints of $e$).
The existence of such a circle contradicts to the condition (ii) of the definition of prime \fkd.
\end{proof}

Theorem~\ref{Theorem:GammaMinimum}
has following obvious consequence.

\begin{corollary}
\label{corol:NoLoop}
If an edge $e$ of a \fkd~$F$ with
a minimal shortcut $\gamma$ is a $\gamma$-edge
then $e$ is not a loop.
\end{corollary}

\begin{remark} \label{rem:numbering}
Theorem~\ref{Theorem:GammaMinimum} implies that the minimal shortcut $\gamma$ of a prime \fkd~$F$
traverses through $\gamma$-regions sequentially one-by-one
without coming back to already visited  regions.
Going from a $\gamma$-region to the next one
$\gamma$ crosses a $\gamma$-edge
which is the only common edge of these two regions.
Therefore, since by definition of the height the interior of $\gamma$ crosses $F$ exactly $\hh(F)$ times,
there are $\hh(F)+1$ pairwise distinct $\gamma$-regions.
Each of them, except  the first one and the last one (the two coincide for \fkd~of the height $0$),
has common edges with two other $\gamma$-regions,
while the first and the last $\gamma$-regions (if they are distinct)
has common edge with one other $\gamma$-region only.
Therefore, there is a natural numbering of $\gamma$-regions of a prime \fkd~$F$ with fixed minimal shortcut $\gamma$:
$\Delta_0,\Delta_1,\ldots,\Delta_{\hh(F)}$,
where $\Delta_0$ and $\Delta_{\hh(F)}$
are $\gamma$-regions adjacent to the beginning and to the end of $F$, respectively.
All other $\gamma$-regions are numbered from $1$ to $\hh(F)-1$
in accordance  with the order in which the shortcut $\gamma$
traverses the regions.
Below we will refer to the numbering as the \emph{canonical numbering}
and to the corresponding numbers of $\gamma$-regions as the their \emph{canonical numbers}.
\end{remark}

\subsection{Types of crossings of flat knotoid diagram}
\label{sec:TypesOfCrossings}

\input{VertexTypeExample}

Let  $x$ is a crossing of \fkd~$F$ with a shortcut $\gamma$.
We will say that the crossing $x$ \emph{has the type $n$}, $0 \leq n \leq 4$,
if $x$ is adjacent to exactly $n$ $\gamma$-edges (counted with multiplicity).
Clearly, the type of a crossing depends on the choice  of a shortcut.
Denote by $c_n(F,\gamma)$ the number of crossings of \fkd~$F$ having the type $n$
with respect to the shortcut $\gamma$.
In Fig.~\ref{Figure:VertexTypeExample} 
we draw a \fkd~and indicate types of all its crossings.

\input{VertexTypeNeightborhoods}

In Fig.~\ref{Figure:VertexTypesNeighborhoods}
we draw a neighbourhoods  of a crossings of all possible types
(as we explain  below, the type~$3$ is impossible).

\begin{lemma}
\label{Lemma:NoFourType}
If $F$ is a prime \fkd~with a minimal
shortcut $\gamma$ then
$c_3(F, \gamma) = c_4(F, \gamma) = 0$.
\end{lemma}

\begin{proof}
Let $x$ be a crossing of $F$.
Suppose $x$ is of the type~$3$.
Then, by definition, exactly $3$ of edges adjacent to $x$ are $\gamma$-edges
while the fourth edge (denote it by $e$) is not.
By the second statement of Theorem~\ref{Theorem:GammaMinimum}
at least one of regions adjacent to $e$ is not a $\gamma$-region.
Hence at least one of regions adjacent to the crossing $x$ is not a $\gamma$-region.
Hence, again by the second statement of Theorem~\ref{Theorem:GammaMinimum},
at least $2$ of edges adjacent to $x$ (counted with multiplicity)
are not $\gamma$-edges,
this contradicts our assumption that  $x$ is of the type~$3$.

Suppose a crossing $x$ is of the type~$4$.
Denote by $e_1,e_2,e_3,e_4$ edges adjacent to $x$.
Then, by definition, all these edges are $\gamma$-edges
and by Corollary~\ref{corol:NoLoop}
no one of them is a loop.
Denote by  $p_i =\gamma \cap e_i,i=1,\ldots,4,$
and let $e_i$ and $p_i$ are numbered in the order
in which $\gamma$ goes through these points.
Note that in all possible situations
we can connect $p_1$ either with $p_3$ or with $p_4$
by an arc not intersecting $F$. Hence the shortcut $\gamma$ is not minimal. This is contradicting to the hypothesis of the lemma.
\end{proof}

\begin{theorem}
\label{Theorem:Equivalence}
Let $F$ be a prime \fkd~with a minimal shortcut $\gamma$. Then following inequalities are equivalent:
\begin{equation} \label{eq:cr geq 2h}
\crn(F) \geq 2 \hh(F)
\end{equation}
and
\begin{equation} \label{eq:c_0 geq c_2}
c_0(F, \gamma) +2 \geq c_2(F, \gamma).
\end{equation}
\end{theorem}

\begin{proof}
By lemma~\ref{Lemma:NoFourType},
$$\crn(F) = c_0(F, \gamma) + c_1(F, \gamma) + c_2(F,\gamma).$$
By Corollary~\ref{corol:NoLoop}
a $\gamma$-edge has distinct endpoints.
Hence the total number of crossings which are endpoint of $\gamma$-edges
is equal to
$2\cdot \hh(F) + 2$
(here the first term corresponds to non-outer $\gamma$-edges, the second one corresponds to two outer edges).
On the  other hand, the same number is equal to
$c_1(F, \gamma) + 2 c_2(F, \gamma)$.
Therefore,
$$2 \hh(F) + 2 = c_1(F, \gamma) + 2 c_2(F, \gamma).$$
Subtracting the last equality from previous one we obtain
$$\crn(F) - 2 \hh(F) -2 = c_0(F, \gamma) -c_2(F, \gamma),$$
hence
$$\crn(F) - 2 \hh(F) = c_0(F, \gamma) + 2 -c_2(F, \gamma).$$
This completes the proof, because the left-hand side of the equality
is equal to the difference between the left-hand side and the right-hand side
of the inequality ~\eqref{eq:cr geq 2h},
while the right-hand side is similarly  connected with the inequality~\eqref{eq:c_0 geq c_2}.
\end{proof}

\subsection{The left and the right border edges}
\label{sec:BorderEdges}

Given a \fkd~$F$ with a shortcut $\gamma$,
an edge $e$ of $F$ is called a \emph{border edge}
if one of regions adjacent to $e$ is a $\gamma$-region
while the other one is not.

\input{RightLeftExample}

We need to partition the set of border edges into two disjoint subsets.
To this end note that
the union of a shortcut  with outer edges
cuts each $\gamma$-region and its boundary into two connected parts,
one of them lies to the left and the other lies to the right of the shortcut $\gamma$
(recall that a shortcut is directed from the beginning of the \fkd~to its end).
A border edge $e$ is called a \emph{left border edge} (resp. a \emph{right border edge})
if it is contained in the left (resp. the right) part
(in the sense above) of the boundary of a $\gamma$-region adjacent to the edge $e$.
Note that by Theorem~\ref{Theorem:GammaMinimum}
if a shortcut is fixed then
the status (either left or right) of a border edge is determined unambiguously.
(In Fig.~\ref{Figure:RightLeftExample}
the edges $e_1$ and $e_2$ are  the left border edge and the right border edge, respectively.)

A crossings of the type~$0$ play an important role in our construction,
so we need to study them more extensive.

A crossing $x$ of the type~$0$ of a \fkd~$F$ with a shortcut $\gamma$
is called:
\begin{itemize}
\item \emph{regular crossing}:
if at least one of edges adjacent to $x$ is not a border edge,

\item \emph{exceptional crossing}:
if all $4$ edges adjacent to $x$ are border edges,

\item \emph{left (resp. right) one-sided exceptional crossing}:
if all $4$ edges adjacent to $x$ are the left (resp. the right) border edges,

\item \emph{two-sided exceptional crossing}:
if $2$ of edges adjacent to $x$ are left border edges
while $2$ other are right.
\end{itemize}

Throughout the terms defined above
are applied to a crossings of the type~$0$ only,
hence we can omit the words ``of the type~$0$'' in corresponding word-combinations.
For example, sometimes we will write ``a one-sided exceptional crossing'' instead of ``a one-sided exceptional crossing of the type~$0$''.

Finally, we define a distance between two regions of a \fkd.
Let $F$ be a \fkd, and $R_1,R_2$ are two its regions.
Denote by $\rho(R_1,R_2)$ the minimal number of intersections
of a simple arc starting inside $R_1$, ending inside $R_2$
and along the way intersecting $F$ transversely in points disjoint from the crossing and the endpoints of $F$.

Note, if $x$ is an exceptional crossing then irrespective of whether it is one-sided or two-sided
exactly $2$ (counted with multiplicity) of regions adjacent to the crossing
are $\gamma$-regions.

\begin{lemma}
\label{Lemma:DistanceOneSide}
Let $F$ be a prime \fkd~with a minimal shortcut $\gamma$,
$x$ is an exceptional crossing of the type~$0$
and $\Delta_1,\Delta_2$ are $\gamma$-regions adjacent to $x$.
Then
\begin{equation*}
\rho(\Delta_1,\Delta_2)=\begin{cases}
1, & \text{if $x$ is two-sided,}\\
2, & \text{if $x$ is one-sided.}
\end{cases}
\end{equation*}
\end{lemma}

\input{TwoPointArc}

\begin{proof}
Note (see Fig.~\ref{Figure:TwoPointArc}),
there is a simple arc $l$ going from $\Delta_1$ to $\Delta_2$
and crossing $F$ twice nearby $x$.
Hence $\rho(\Delta_1,\Delta_2) \leq 2$.
Assume $\rho(\Delta_1,\Delta_2) =0$. It means $\Delta_1 =\Delta_2 =\Delta$.
In this case we can close the arc $l$ up by $\Delta$.
The resulting circle meets $F$ twice
and bounds two disks  which contain a crossings.
The existing of such a circle contradicts to the condition (i) of the definition of prime \fkd.
Therefore,
$$1 \leq \rho(\Delta_1,\Delta_2) \leq 2.$$

Let the crossing $x$ is one-sided. It is sufficient  to show $\rho(\Delta_1,\Delta_2) \not=1$.
Assume the contrary. In this case the shortcut $\gamma$
traverses from $\Delta_1$ to $\Delta_2$ with one intersection of $F$ only.
Thus we can close the arc $l$ up  with exactly one additional intersection of $F$.
The resulting circle meets $F$ exactly three times,
while the number should be even. It's
because both endpoints of the \fkd~lie in the same disk, bounded by the circle.
Hence, the number of goings into the disk should be equal to the number of goings out.

Let the crossing $x$ is two-sided. It is sufficient to show $\rho(\Delta_1,\Delta_2) \not=2$.
Again assume the contrary. Now closing the arc $l$ up
we obtain a circle,  which separates the endpoints of $F$
and meets $F$ exactly $4$ times,
while in this case the number of intersections should be odd.
\end{proof}

\subsection{The left and right border chains}
\label{sec:BorderChains}

Given a \fkd~$F$ with a shortcut $\gamma$,
the \emph{$\gamma$-domain} of $F$ (denoted by $R_{\gamma}$) is defined to be
the union of all $\gamma$-regions with the interior of all $\gamma$-edges  and both endpoints of $F$.
Informally speaking, we clean all $\gamma$-edges from $F$ (including both outer edges),
as a result, $\gamma$-regions amalgamate to an $R_{\gamma}$.

The set $S^2 \setminus R_{\gamma}$ is the union
of all regions which are not a $\gamma$-regions
with all edges which are not $\gamma$-edges.
It is easy to show that $R_{\gamma}$ is open and connected.
If $F$ is prime, then by Theorem~\ref{Theorem:GammaMinimum}
$R_{\gamma}$ is homeomorphic to an open disk.
In this case each exceptional crossing (both one-sided and two-sided) 
is a point of self-tangency of $\partial R_{\gamma}$.
The union of $\gamma$ with the outer edges is a diameter of $R_{\gamma}$
viewed as a disk.

Note, $\partial R_{\gamma}$ consists of all border edges.
Therefore, in the case of prime \fkd~
we can regard $\partial R_{\gamma}$ as
the closed path in $F$, which goes  exactly one time along each border edge. 
We denote the path by $P_{\gamma}$.
The path $P_{\gamma}$ can be divided into two parts
by two crossings which are connected by outer edges with the endpoints of $F$.
One of these parts is formed by all left
 and other one is formed by all right border edges.
It follows  from the fact
that the union of $\gamma$ with outer edges divides $R_{\gamma}$
and $\partial R_{\gamma}$ into two parts,
one of which lies to the left of $\gamma$
while the other one lies to the right.

\input{RightChainExample}

A \emph{left (resp. right) border chain} is define to be a sequence of the left (resp. the right) border edges,
forming a connected subpath of $P_{\gamma}$.
In Fig.~\ref{Figure:RightChainExample}
a left border chain $E_l$ and a right border chain $E_r$ of a \fkd~are shown.
Below we denote such a chains by $E =\{e_1,\ldots,e_n\}$
where edges $e_i,e_{i+1}$ are neighbouring in the path $\partial R_{\gamma}$.
For simplicity we will think that edges involving
in a chain are directed according specified ordering of edges.
So we can say about the beginning and the end of a chain
(the latter are called the endpoints of the chain)
and about crossings which a chain passes through

\begin{remark} \label{rem:Ambiguity}
To prevent an ambiguity in using terminology, we make following remarks.
\begin{enumerate}
\item A border chain is understood as an ordered set of border edges.
In particular, it means that two border chains which are disjoint in the sense above
can have non-empty intersection if they are viewed as the subsets of $S^2$.
\item If a crossing $x$ is an exceptional one-sided crossing,
then, as we mentioned above, $x$ is a point of self-tangency of $\partial R_{\gamma}$.
In this case we think that $x$ has two distinct entries in $P_{\gamma}$.
Therefore, a left (resp. right) border chain can pass through an one-sided left (resp. right) crossing $0,1$ or $2$ times.
\item We think that a border chain does not pass through its endpoints
(even in the case when the endpoints of the chain
are two entries of the same crossing).
\end{enumerate}
\end{remark}

From now we focus on a specific border chains
which play the key role in our consideration.

A left (resp. right) border chain $E$ is called \emph{true} 
if the endpoints of $E$ are 
either of the type~$2$
or one-sided left (resp. right) exceptional crossing of the type~$0$
(the situation in which an endpoint is of the type~$0$ while the other one is of the type~$2$ is allowed).

\begin{lemma}
\label{Lemma:NonEndedChain}
Let $F$ is a prime \fkd~with a shortcut $\gamma$
and $E$ is a true border chain satisfying following conditions:
\begin{enumerate}
\item $E$ do not contain  a true border chain distinct from $E$;
\item No one of endpoints of $E$ is adjacent to an outer edge;
\item $E$ passes through not more than $1$ two-sided exceptional crossing.
\end{enumerate}
Then $E$ passes through at least $1$ regular crossing.
\end{lemma}

\begin{proof}
Let $E =\{e_1,\ldots,e_n\},n \geq 1,$ be a left true border chain
(in the case of a true right  border chain the proof is completely analogous).

Denote by $x$ and $y$ the beginning and the end of $E$, respectively.
Let the numbering of edges in the chain $E$ is such that
going along the left border edges from the beginning of $F$ to its end
(or more precisely, going along the part of $P_{\gamma}$ consisting of left border edges
from the crossing adjacent to the first outer edge of $F$
to the crossing adjacent to the last outer edge of $F$)
we meet $x$ before $y$.

Assume $E$ do not passes through a regular crossing.
Then all crossings in the chain $E$ except its endpoints
are either of the type~$1$ or two-sided exceptional crossing.
That is because by the first condition of~Lemma~\ref{Lemma:NonEndedChain}
$E$ does not pass through neither a crossing of the type~$2$
nor a one-sided exceptional crossing.
Denote by $k$ the number of crossings of the type~$2$ through which $E$ passes.
By hypothesis, $k$ is equal either to $0$ or to $1$.

\input{ChainExeptions}

{\bf The case $k=0$}.
Denote by $\Delta(e_i)$ and $N(e_i),i=1,\ldots,n,$
the $\gamma$-region adjacent to the edge $e_i$
and its canonical number (see Remark~\ref{rem:numbering}), respectively.
Let $\Delta_x,\Delta_y$ denote the $\gamma$-regions
having numbers $N(e_1)-2$ and $N(e_n)+2$, respectively
(see Fig.~\ref{Figure:ChainExeptions} on the left).
The existence of  $\Delta_x$ and $\Delta_y$ satisfying the latter condition
is clear in the case of crossing of the type~$2$,
and follows from Lemma~\ref{Lemma:DistanceOneSide}
in the case of exceptional one-sided crossing.
Note that the shortcut $\gamma$ going from $\Delta_x$ to $\Delta_y$ crosses $F$ $n+3$ times.
At the same time since all edges in the chain $E$ are left border edges,
there exists a path (the arc $l$ in Fig.~\ref{Figure:ChainExeptions} on the left)
going from $\Delta_x$ to $\Delta_y$,
which crosses $F$ $n+1$ times.
This contradicts the minimality of the shortcut $\Gamma$.

{\bf The case $k=1$}.
Denote by $z$ the two-sided exceptional crossing
which the chain $E$ passes through.
Then we have two left border edges which are adjacent to $z$
and, by our assumption,  the edges are involved in the chain $E$.
Denote them by $e_s$ and $e_{s+1}$ (see Fig.~\ref{Figure:ChainExeptions} on the right).
Let $\Delta_x,\Delta_y,\Delta(e_i)$ are as above.
By Lemma~\ref{Lemma:DistanceOneSide},
there are two distinct $\gamma$-regions adjacent to the crossing $z$,
these are $\Delta(e_s)$ and the other one which we denote by $\Delta_z$.
By Lemma~\ref{Lemma:DistanceOneSide}
$\rho(\Delta_z,\Delta(e_s))=1$,
i.e., the canonical number of the region $\Delta_z$ (see Remark~\ref{rem:numbering})
either is less by $1$ or is greater by $1$ than the canonical number of the region $\Delta(e_s)$.
Assume the number of $\Delta_z$ is greater than the number of $\Delta(e_s)$.
In this case there exists a path which goes from $\Delta_x$ to $\Delta_z$
which intersects $f$ $s+1$ times.
It means $\rho(\Delta_x,\Delta_z) \leq s+1$.
At the same time, since $\gamma$ is minimal $\rho(\Delta_x,\Delta(e_s)) =s+1$, the number of $\Delta_z$ is less than or equal to the number of $\Delta(e_s)$.
This contradicts our assumption that the number of $\Delta_z$ is greater than the number of $\Delta(e_s)$.
The arguments in the case when the number of $\Delta_z$ is less than the number of $\Delta(e_s)$
are analogous to the arguments above.
The only difference is that in the case it is necessary to compare
the distance from $\Delta_y$ to the $\gamma$-regions adjacent to $z$.
\end{proof}

Lemma~\ref{Lemma:NonEndedChain}
can not be extended directly
to border chains of which endpoints are adjacent to outer edges.
To do this an additional condition is required.

\input{LeftRightVertex}

Let the crossing $x$, which adjacent to  an outer edge of a \fkd, is a crossing of the type~$2$.
Denote $\gamma$-edges adjacent to $x$ by $e_1$ and $e_2$,
where $e_1$ is the outer edge of the \fkd.
The crossing $x$ is called the \emph{left-sided} (resp. the \emph{right-sided}) crossing of the type~$2$,
if starting at $x$ and going along $e_2$
we reach $\gamma$ from the left (resp. from the right)
(see Fig.~\ref{Figure:LeftRightVertex} on the left (resp. on the right)).

\begin{lemma}
\label{Lemma:EndedChain}
Let $F$ is a prime \fkd~with a shortcut $\gamma$,
and $E$ is a true border chain satisfying following conditions:
\begin{enumerate}
\item $E$ do not contain  a true border chain distinct from $E$;
\item If $E$ is a left (resp. right) border chain and an endpoint of $E$ is adjacent to an outer edge of $F$
then the endpoint is a left-sided (resp. right-sided) crossing of the type~$2$;
\item $E$ passes through no two-sided exceptional crossing.
\end{enumerate}
Then $E$ passes through at least $1$ regular crossing.
\end{lemma}

\begin{proof}
Let $E$ be a true left border chain
(the proof in the case of the right border chain is completely analogous).

Firstly consider the case in which both endpoints of $E$ are adjacent to outer edges of $F$.
We assume $E$ does not pass through a regular crossing
and show that it is impossible.
By the first condition of Lemma~\ref{Lemma:EndedChain}
$E$ does not pass through neither a crossing of the type~$2$ nor a one-sided exceptional crossing.
By the third condition $E$ does not pass through two-sided exceptional crossings.
Hence all crossings in $E$ except its endpoints are of the type~$1$.
By hypothesis, both endpoints of $E$ are left-sided crossings of the type~$2$.
Hence the union of $E$ with outer edges of $F$ forms a path,
which goes from the beginning of $F$ to its end
and crosses the rest part of the diagram transversely.
Thus the diagram (which, by definition, is a generic immersion of the segment into $S^2$)
is indeed a generic immersion of a disconnected $1$-manifold,
i.e., in this case the diagram in question is not a \fkd.

If both endpoints of $E$ are not adjacent to outer edges,
then required property follows from Lemma~\ref{Lemma:NonEndedChain}.
So it is remains to consider the case, when exactly one of endpoints (say the beginning)
is adjacent to an outer edge of $F$.
Then, by hypothesis, the endpoint is left-sided crossing of the type~$2$.
In such a situation we can use the same trick as in the case of $k=0$ in the proof of Lemma~\ref{Lemma:NonEndedChain}.
\end{proof}

\begin{lemma}
\label{Lemma:InnerOneSideVertices}
Let $F$ be a prime \fkd~with a minimal shortcut $\gamma$,
and $E$ be a left (resp. right) border chain
which starts and ends at the same one-sided exceptional left (resp. right) crossing of the type~$0$.
Then
\begin{enumerate}
\item $E$ passes through not more than one crossing of the type $2$,
\item If $E$ passes through an exceptional crossing of the type~$0$
distinct from its endpoints,
then $E$ passes through the crossing twice.
\end{enumerate}
\end{lemma}

\input{EndedVertexZero}

\begin{proof}
Denote by $x$ the crossing at which $E$ starts and ends. 
By Lemma~\ref{Lemma:DistanceOneSide}
the distance between two $\gamma$-regions adjacent to $x$ is equal to $2$.
Hence $\gamma$ traversing from one of these $\gamma$-regions to another
crosses exactly $2$ $\gamma$-edges $g_1,g_2$.
Consider an embedded circle $C$
(see Fig.~\ref{Figure:EndedVertexZero})
which passes through $x$ and crosses the edges $g_1,g_2$ transversely in their internal points.
Since $C$ and $F$ share exactly $3$ points,
there is no $\gamma$-edges except $g_1,g_2$
adjacent to a crossing in $E$.
Hence, if $g_1,g_2$  are adjacent to two distinct crossings 
which the chain $E$ passes through,
then both these crossings are of the type~$1$.
If $g_1,g_2$ are adjacent to the same crossing in $E$,
then the crossing is of the type~$2$,
and there is no more crossings of the type~$2$ in $E$.
This completes the proof of the first part of Lemma~\ref{Lemma:InnerOneSideVertices}.

Denote by $D$ the disk bounded by the circle $C$, which contains the chain $E$.
Note the disk does not contain border edges except edges forming the chain $E$.
The fact is a consequence of following two observations:
\begin{enumerate}
\item The set $\partial R_{\gamma}$ is connected.
\item $\partial R_{\gamma} \cap C =\{ x\}$,
hence  no border chain but $E$ crosses the circle $C$.
\end{enumerate}
Therefore, if an exceptional crossing $y$ lies inside $D$
then all $4$ border edges adjacent to $y$ are involved in $E$,
thus, as required,  $E$ passes $y$ twice.
\end{proof}

\subsection{The lower bound for prime flat knotoid diagrams}
\label{sec:ProofForPrime}

The theorem below is the key step in the proof
of Theorem~\ref{theorem:MainResult}.
It states an inequality~\eqref{eq:crnF geq2hF}
which is like to inequality~\eqref{eq:Main}.
The difference between them is that in~\eqref{eq:crnF geq2hF}
we compare two characteristics of the same diagram
while in~\eqref{eq:Main}
we deal with two characteristics of an equivalence class
which can be reached on two distinct representatives.

\begin{theorem}
\label{Theorem:VertexTypes}
If $F$ is a prime \fkd~then
\begin{equation}
\label{eq:crnF geq2hF}
\crn(F) \geq 2 \hh(F).
\end{equation}
\end{theorem}

\begin{proof}
Fix a minimal shortcut $\gamma$ of the \fkd~$F$.
All objects such as $\gamma$-edges, $\gamma$-regions, left/right border chains and so on
are considered with respect to the shortcut.
Therefore, for shortness we can sometimes omit the letter $\gamma$ in using notation.

We need an additional definitions.
Given a border chain starting and ending at the same one-sided exceptional crossing $x$
and a crossing $y$ which the chain passes through,
in this case we will say that the crossing $x$ \emph{frames} the crossing $y$,
or, equivalently, that the crossing $y$ is \emph{framed} by the crossing $x$.
Recall that, by definition, a border chain can not contain left and right border edges at the same time.
Thus a one-sided crossing $x$ determines exactly one boundary chain
which starts and ends at $x$. Therefore, $x$ frames all crossings the chain passes through.
Since in our terminology a border chain does not passes through its endpoints,
$x$ does not frame itself.
If $x$ frames $y$ and $y$ is also one-sided exceptional crossing,
then by Lemma~\ref{Lemma:InnerOneSideVertices}
the border chain starts and ends at $y$ is contained in the border chain
starting and ending at $x$.
So the definition above does not depend on what entry of $y$
in $P_{\gamma}$ we use.
An exceptional one-sided crossing $x$ is called \emph{maximal}
if there is no crossing framing $x$.

Denote by $\mathcal{C}_0$ the set consisting of all maximal exceptional one-sided crossings
and by $\mathcal{C}_2$ the set consisting of all crossings of the type~$2$
which are not framed by an exceptional one-sided crossing.

The crossings involved in $\mathcal{C}_0 \cup \mathcal{C}_2$
divide the path $P_{\gamma}$ into pairwise disjoint parts (chains)
(such a chains are disjoint if they do not share an edge, see Remark~\ref{rem:Ambiguity}).
Most of them are either left or right border chains,
but one or two parts can contain a left and right border edges at the same time
and thus are not neither left nor right  border chains.
The latter situation occurs when a crossing adjacent to an endpoint (or both such a crossings)
does not involved in $\mathcal{C}_0 \cup \mathcal{C}_2$.
It is necessary to recall that if $x \in \mathcal{C}_0$ then $P_{\gamma}$
passes through $x$ twice, and there are three chains
adjacent to $x$: one of them starts and ends at $x$ and two other have $x$ as one of their  endpoints.

Firstly we will prove that
there exists a minimal shortcut $\gamma$ of $F$ for which
\begin{equation} \label{eq:c_0 geq q}
c_0(F,\gamma) \geq q -2
\end{equation}
where $q$ is the number of pairwise disjoint (in the sense above) parts into which
$\mathcal{C}_0 \cup \mathcal{C}_2$ divides $P_{\gamma}$.
To this end we will construct a map $Z$
carrying the chains $E_1,\ldots,E_q$ (except one or two)
to crossings of the type~$0$
such that the images of any two distinct chains are distinct.

Clearly, \eqref{eq:c_0 geq q}
holds if $q \leq 2$.
So below we assume $q >2$.

Denote by $u,v \in \partial R_{\gamma}$ crossings adjacent to the first and to the last edges of $F$, respectively.
(We emphasize that $u$ and $v$ are  not the beginning and the end of $F$, they are the other endpoints of outer edges.)
Below we consider four cases depending on
whether $u$ and $v$ are elements of $\mathcal{C}_0 \cup \mathcal{C}_2$.
Note $u$ and $v$ can not be of the type~$0$, because an outer edge is by definition an $\gamma$-edge.
Hence $u$ (resp. $v$) belongs to $\mathcal{C}_0 \cup \mathcal{C}_2$
if and only if the crossing belongs to $\mathcal{C}_2$.

{\bf 1. $u \not\in \mathcal{C}_2, v \not\in \mathcal{C}_2$.}
In this case the crossings $u$ and $v$ lie inside some chains under consideration.
The chains are not necessary distinct.
Thus we have not less than $q-2$ chains
$E_1,\ldots,E_s,s \geq q-2,$
which do not pass through neither $u$ nor $v$.
No of these chains can contain left and right border edges at the same time
(that is because $u$ and $v$ are the only two crossings
which are adjacent to both left and right border edges at the same time).
Therefore, each of chains $E_1,\ldots,E_s$ is either left or right border chain.

The set of chains $E_1,\ldots,E_s$
can be decomposed into three (possibly empty) subsets (types):
\begin{itemize}
\item [(i)] The chains starting and ending at the same crossing belonging to $\mathcal{C}_0$.
\item [(ii)] The chains passing through $2$ or more two-sided exceptional crossings.
\item [(iii)] The chains passing through not more than $1$ two-sided exceptional crossing.
\end{itemize}
Now we define $Z(E_j),1 \leq j \leq s$,
i.e., we assign a crossing of the type~$0$ (the crossing can be both exceptional and regular)
to each of chains in question.
The rule of the assigning  depends on the type of the chain.
\begin{equation} \label{def:Z}
\begin{tikzpicture}[baseline]
\node [text width=0.8\textwidth]
{
The crossing of the type~$0$ which the map $Z$ assign to a chain is given by following rule:
\begin{itemize}
\item A chain of the type~(i) $\to$  The one-sided exceptional crossing at which the chain starts and ends.
\item A chain of the type~(ii) $\to$ A  two-sided exceptional crossings which the chain passes through.
\item A chain of the type~(iii) $\to$ A regular crossing of the type~$0$ which the chain passes through.
\end{itemize}
};
\end{tikzpicture}
\end{equation}

It is necessary to explain how to make the map injective.
Firstly note that we assign a crossings of different types (one-sided exceptional, two-sided exceptional and regular)
to chains of different type
(type (i),(ii) and (iii), respectively).

(i) Two different chains of the first type can not starts and ends
at the same one-sided exceptional crossing,
because if such two chains exist, then the union of these chains is equal to $P_{\gamma}$.
Hence at least one of the chains passes through $u$ or $v$,
while such a chains are excluded from those for which we define the map $Z$.

(ii) A two-sided exceptional crossing can lie in two chains of this type
(one of them is left border chain and the other is right one).
Let we have $l$ and $r$ left and right border chains of the type~(ii), respectively.
Since each chain in question passes through at least $2$
two-sided exceptional crossings,
the total number of such a crossings is greater than or equal to $\max(2l,2r)$.
Hence an injectivity can be established because the total number
of the chains is equal to $l+r$.

(iii) In the case under consideration there are no chains starting at $u$ or ending at $v$. Hence all chains of the type~(iii) (if any) satisfy the conditions of Lemma~\ref{Lemma:NonEndedChain}. Hence each of the  chain passes through a regular crossing of the type~$0$
which is distinct from the endpoints of the chain
and thus does not belong to any other chain under consideration.

{\bf 2. $u \in \mathcal{C}_2, v \not\in \mathcal{C}_2$.}
Let the chains $E_1,E_2$ are the left and the right chains adjacent to $u$, respectively,
and the chain $E_3$ passes through $v$.
The crossing $u$ is either left or right (say left) crossing of the type $2$.
There are three possibilities.

{\bf 2.1.}
$E_1 =E_3$, i.e.,
the chain $E_1$ passes through the crossing $v$.
In this case all chains distinct from $E_1,E_2$(if any) are right border chains.
The map $Z$ can be defined using the rule~\eqref{def:Z} as follows.
The chains $E_1$ and $E_2$ are excluded.
All other chains are regarded as a chains of the type either~(i) or~(ii) or~(iii)
with respect to the definition above.
The injectivity of the map is obvious because we have right border chains only.

{\bf 2.2.}
The chain $E_1$ do not pass through the crossing $v$
(i.e., $E_1 \not=E_3$)
and do not pass through a two-sided exceptional crossing
(it can not pass through one-sided exceptional crossing by construction).
In this case $E_1$ satisfy the conditions of Lemma~\ref{Lemma:EndedChain}.
Hence the chain passes through a regular crossing.
Therefore, we can define the map $Z$ as follows.
Two chains ($E_2,E_3$) are excluded and the chain $E_1$ is viewed as a chain of the type (iii).
All other chains (if any) are regarded in accordance with the rule~\eqref{def:Z}.
The arguments concerning the injectivity of resulting map
coincide with those in the case~1 above.

{\bf 2.3.}
The chain $E_1$ do not pass through the crossing $v$
(i.e., $E_1 \not=E_3$)
and passes through at least $1$ one-sided exceptional crossing of the type~$0$.
In this case we again can use the rule~\eqref{def:Z} regarding $E_1$ as a chain of the type~(ii)
but it is necessary to explain an injectivity of the resulting map.
Denote by $l$ and $r$ the numbers of left and right chains
of the type (ii), respectively.
Thus we have $l+r+1$ chains of the type (ii)
(here the chain $E_1$ is added).
So left chains pass through at least $2l+1$ two-sided exceptional crossings,
while right chains pass through $2r$ such a crossings.
But left and tight chains pass through the same two-sided exceptional crossing
(left chains pass through them on the left while right chains pass on the right),
thus the total number of such a crossings is greater than or equal to $\max(2l+1,2r)$.
Hence a injectivity can be established in both cases: $l \leq r$ and $l >r$.

{\bf 3. $u \not\in \mathcal{C}_2, v \in \mathcal{C}_2$.}
The case is completely analogous to the case~2.

{\bf 4. $u \in \mathcal{C}_2, v \in \mathcal{C}_2$.}
Denote by $E_1,E_2,E_3,E_4$ the border chains
such that $E_1,E_2$ are adjacent to $u$, $E_3,E_4$ are adjacent to $v$,
$E_1,E_3$ are left, $E_2,E_4$ are right.
Since $u$ and $v$ are crossings of the type~$2$,
each of them can be either left-sided or right-sided,
and it is necessary to consider following situations.

{\bf 4.1. $u$ is left-sided while $v$ is right-sided.}
(The case when $u$ is right-sided, $v$ is left-sided is completely analogous.)

{\bf 4.1.1.}
Let $E_1=E_3$ and $E_2=E_4$.
Then we have two chains only, i.e., $q=2$
hence \eqref{eq:c_0 geq q}~holds.

{\bf 4.1.2.}
Let $E_1=E_3$, $E_2 \not=E_4$
and $E_1$ passes through a two-sided exceptional crossing.
We define the map $Z$ excluding the chains $E_2$ and $E_4$.
To $E_1$ we assign the two-sided exceptional crossing which $E_1$ passes through.
All other chains (if any) are regarded in accordance with the rule~\eqref{def:Z}.
In this case the injectivity of the map is obvious.
The situation when $E_1 \not=E_3$, $E_2 =E_4$
and $E_2$ passes through a two-sided exceptional crossing
is completely analogous.

{\bf 4.1.3.}
Let $E_1 =E_3$, $E_2 \not=E_4$
 and $E_1$ do not passes through a two-sided exceptional crossing.
Since $E_1$ connects $u$ and $v$
the chain contains all left border edges of $F$.
Hence there is no two-sided exceptional crossings in $F$
(otherwise $E_1$ contains two left border edges adjacent to the crossings
and thus passes it through).
Consequently the chain $E_4$ satisfies the conditions of Lemma~\ref{Lemma:EndedChain}
and we can assign to $E_4$ the regular crossing which  the chain passes through.
All other chains (if any) can be regarded in accordance with the rule~\eqref{def:Z}.
The injectivity of the resulting map is obvious.
The situation $E_1 \not=E_3$, $E_2 =E_4$
and $E_2$ do not passes through a two-sided exceptional crossing
is completely analogous.

{\bf 4.1.4.}
Let $E_1 \not=E_3$, $E_2 \not=E_4$.
This case is like to the cases $2$ and $3$.
Now we exclude the chains $E_2$ and $E_3$.
For the chains $E_1$ and $E_4$ the map $Z$ is defined by the same way
as for the chain $E_1$ in the case 2.2.
To prove the injectivity of the correspondence it is sufficient to estimate the number
of two-sided exceptional crossings.
Now the number is greater than or equal to
$\max(2l+1,2r+1)$
hence an injectivity can be established.

{\bf 4.2. Both $u$ and $v$ are right-sided crossings of the type~$2$.}
(The case when both these crossings are left-sided are completely analogous.)

{\bf 4.2.1.}
Let $E_2=E_4$ and the chain do not pass through neither two-sided exceptional nor regular crossing.
In this case the union of $E_2$ with outer edges is a path
which intersects the rest part of $F$ transversely.
This contradicting the definition of \fkd, because
$F$ is a generic immersion not of a segment but  a disconnected $1$-manifold.

{\bf 4.2.2.}
Let $E_2=E_4$ and the chain satisfies
at least one of conditions:
\begin{enumerate}
\item It passes through a regular crossing;
\item It passes through $2$ or more two-sided exceptional crossings.
\end{enumerate}
The map $Z$ is define as follows.
The chains $E_1$ and $E_3$ are excluded.
To $E_2$ we assign either the regular crossing (if the first condition holds)
or a two-sided exceptional crossing (otherwise).
All other chains (if any) are regarded in accordance with the rule~\eqref{def:Z}.
The injectivity of the resulting map is obvious.

Therefore, in the case $E_2=E_4$ it remains
to consider the only situation
when $E_2$ do not pass through a regular crossing and 
it passes through exactly $1$ two-sided exceptional crossing.
This will be done in the case 4.2.4 below.

{\bf 4.2.3.}
Let $E_2 \not=E_4$ and
each of these chains satisfies at least one of following conditions:
\begin{enumerate}
\item The chain passes through a regular crossing;
\item The chain passes through $2$ or more two-sided exceptional crossings;
\item The chain do not pass through a two-sided exceptional crossing.
\end{enumerate}
Note if the last condition holds,
than the chain satisfies the conditions of Lemma~\ref{Lemma:EndedChain}
and thus it passes through a regular crossing.
The map $Z$ is defined as follows.
The chains $E_1$ and $E_3$ are excluded.
To the chain $E_j, j \in \{2,4\},$ we assign
either the regular crossing (if the first condition holds)
or a two-sided exceptional crossing (otherwise).
All other chains (if any) are regarded in accordance with the rule~\eqref{def:Z}.
The injectivity of the map follows from the fact
that the number of two-sided exceptional crossings in $F$
is greater than or equal to 
$\max(2l,2r)$,
where $l$ and $r$ are the numbers of left and right chains which pass through more than $1$ two-sided exceptional crossing.

{\bf 4.2.4.}
Let exactly one of the chains $E_2,E_4$ (say $E_2$) do not satisfies to all three conditions
listed in the previous case,
while the other chains either coincides with the first
or satisfies at least one of these conditions.
Then $E_2$ passes through exactly $1$ two-sided exceptional crossing
and do not pass through a regular crossing.
The map $Z$ is defined as in the previous case
except that to $E_2$ we assign the two-sided exceptional crossing which the chains passes through.
The injectivity of the map follows from the fact
that in the case under consideration the number of two-sided exceptional crossings is greater than of equal to
$\max(2l,2r+1)$,
where the term $1$ in the expression $2r+1$ corresponds to the two-sided exceptional crossing which $E_2$ passes through.

{\bf 4.2.5.}
Let $E_2 \not=E_4$ and
each of  these chains passes through exactly $1$ two-sided exceptional crossing
and do not pass through a regular crossing.
In this case we will replace the shortcut $\gamma$ with a new shortcut $\gamma'$
which is also minimal and gives a situation
satisfying the conditions of the case 4.1.
The shortcut is defined as follows.

Denote by $x$ the two-sided crossing of the type~$0$ which $E_2$ passes through
and by $\Delta_l$ and $\Delta_r$ the left and the right $\gamma$-regions adjacent to the crossing $x$, respectively.
By Lemma~\ref{Lemma:DistanceOneSide}
$\rho(\Delta_l,\Delta_r)=1$.
Since, by hypothesis, the chain $E_2$ is right border chain
and it does not pass through a regular crossing of the type~$0$,
$\rho(\Delta_0,\Delta_l)=k$, 
where $k$ is the number of edges in $E_2$ between the crossings $u$ and $x$.
Thus $\rho(\Delta_0,\Delta_r)$ is equal either to $k+1$ or to $k-1$.

Assume $\rho(\Delta_0,\Delta_r)=k-1$.
Let the chain $E_2$ consists of border edges $\{e_1,\ldots,e_n\},n \geq k+1,$
and $y$ is the end of the chain.
Since $y$ is the end of the chain $y \in \mathcal{C}_0 \cup \mathcal{C}_2$.
If $y$ is a crossing of the type~$2$,
then three $\gamma$-regions adjacent to $y$ are
$\Delta_{n-1},\Delta_n,\Delta_{n+1}$.
Here we use the canonical numbering of $\gamma$-regions (see Remark~\ref{rem:numbering}).
If $y$ is a two-sided exceptional crossing,
then  $y$ is adjacent to exactly two $\gamma$-regions $\Delta_{n-1}$ and $\Delta_{n+1}$.
The chain $\{e_{k+1},\ldots,e_n\}$ do not pass through a regular crossing of the type~$0$.
Hence there exists an arc going from $\Delta_r$ to $\Delta_{n+1}$
which intersects $F$  $n-k+1$ times.
Joining  the arc with initial part of $\gamma$ we obtain an arc
going from $\Delta_0$ to $\Delta_{n+1}$
which intersects $F$ in $k-1+n-k+1=n$ points.
This contradicts the minimality of $\gamma$.

Therefore, $\Delta_r =\Delta_{k+1}$.
To obtain a new shortcut $\gamma'$ we replace
the initial part of $\gamma$ connecting $u$ with an internal point of the region $\Delta_{k+1}$
with another arc starting at $u$ then going parallel $\{e_1,\ldots,e_k\}$ on the right
and ending at the same internal point of $\Delta_{k+1}$.
Since $\{e_1,\ldots,e_k\}$ do not pass through a regular point of the type~$0$
the new arc intersects with $F$ in $k+1$ points.
Hence $\gamma'$ and $\gamma$ have the same number of intersections with $F$
thus $\gamma'$ is minimal also.
At the same time with respect to $\gamma'$ the crossing $u$ is the left-sided crossing of the type~$2$,
i.e., we obtain the situation satisfying  the conditions of the case 4.1.

The proof of inequality~\eqref{eq:c_0 geq q}
is complete. In the case~4.2.5
it is necessary to modify the shortcut.
In all other cases the inequality holds for an arbitrary minimal shortcut.

It remains to show  that
\begin{equation} \label{eq:q geq c_2}
q \geq c_2(F,\gamma).
\end{equation}
If so the inequality~\eqref{eq:c_0 geq q}
implies
$$c_0(F, \gamma) + 2 \geq c_2(F, \gamma)$$
and the inequality~\eqref{eq:crnF geq2hF}
holds by Theorem~\ref{Theorem:Equivalence}.

To see~\eqref{eq:q geq c_2}
recall that $q$ is the number of parts into which the elements of the set
$\mathcal{C}_0 \cup \mathcal{C}_2$ divide the closed path $P_{\gamma}$.
Thus $q=|\mathcal{C}_0 \cup \mathcal{C}_2|$
which (since these two sets are disjoint) is equal to
$|\mathcal{C}_0| +|\mathcal{C}_2|$.
The set of crossings of the type~$2$ can be decomposed into two disjoint subsets
$C'_2 \cup C''_2$.
The first subset consists of such a crossings which are not framed by a one-sided exceptional crossing,
hence, by definition, $C'_2 =\mathcal{C}_2$.
The second subset consists of such a crossings of the type~$2$ which are framed by a one-sided exceptional crossing
and thus by a maximal one-sided exceptional crossing.
By Lemma~\ref{Lemma:InnerOneSideVertices}
a one-sided exceptional crossing can not frame more than $1$ crossing of the type~$2$. Hence
$|C''_2| \leq |\mathcal{C}_0|$.
Therefore
$$q=|\mathcal{C}_0| +|\mathcal{C}_2| \geq |C'_2| +|C''_2| =c_2(F,\gamma).$$
This completes the proof of Theorem~\ref{Theorem:VertexTypes}.
\end{proof}

\subsection{The lower bound for  an arbitrary flat knotoid diagram}
\label{sec:ProofForArbitrary}

\begin{theorem}
\label{Theorem:NonPrimeFlatKnotoids}
The inequality~\eqref{eq:crnF geq2hF}
holds for an arbitrary \fkd~$F$.
\end{theorem}

\begin{proof}
We proceed by induction on $\crn(F)$.

\input{FlatKnotoid1}

{\bf Let $\crn(F) = 1$.}
 There exists exactly one \fkd~$F$ for which
$\crn(F) = 1$
(see Fig.~\ref{Figure:FlatKnotoid1}). 
The height of the \fkd~is equal to $0$,
hence~\eqref{eq:crnF geq2hF}
holds.

{\bf Let $\crn(F) >1$.}
If $F$ is prime then \eqref{eq:crnF geq2hF}~holds
by Theorem~\ref{Theorem:VertexTypes}.

If $F$ is not prime, 
then at least one of conditions (i),(ii) in the definition of a prime \fkd\  
(see Section~\ref{sec:Prime})
does not hold,
i.e., there is an embedded circle $C$
which intersects  $F$ in $1$ or $2$ points
and such that each disk bounded by $C$ contains at least one crossing.

\input{NonPrimeReduction0102}

If the circle $C$ intersects $F$ in exactly $2$ points
(i.e., the condition~(i) does not hold),
then one of disks bounded by $C$ (denote it by $D$) does not contains the endpoints of $F$.
Consider the \fkd~$F'$ obtaining from $F$
by contracting the disk $D$ into a point
or, equivalently, by replacing the fragment inside $D$ with a simple arc connecting the same points in $\partial D$
(see Fig.~\ref{Figure:NonPrimeReduction0102} on the left).
Then we have (since there is a crossing inside $D$ and a minimal shortcut can be pushed outside $D$)
$$\crn(F) > \crn(F'), \quad \hh(F) = \hh(F').$$
By induction assumption
$\crn(F')\geq 2 \hh(F')$,
 hence
$$\crn(F) > \crn(F') \geq 2 \hh(F') = 2 \hh(F).$$

If the circle $C$ intersects $F$ in exactly $1$ points
(i.e., the condition(ii) does not hold),
then $C$ cuts $F$ into two non-trivial \fkd~$F_1$ and $F_2$
which lie inside different disks bounded by $C$.
More precisely, the \fkd~$F_i,i=1,2,$ is obtained as a result of a contracting into a point
the disk $D_i$ where $D_1,D_2$ are the disks into which the circle $C$ cuts the sphere $S^2$
(see Fig.~\ref{Figure:NonPrimeReduction0102} on the right).
Then
\begin{center}
$\crn(F) = \crn(F_1) + \crn(F_2)$,

$\crn(F) > \crn(F_i)$, $i = 1, 2$,

$\hh(F) = \hh(F_1) + \hh(F_2)$.
\end{center}
By induction assumption
$\crn(F_i) \geq 2 \hh(F_i),i=1,2$, 
hence
$$\crn(F) =
\crn(F_1) + \crn(F_2) \geq 2 \hh(F_1) + 2 \hh(F_2) = 2 \hh(F).$$
\end{proof}

\subsection{The proof of the Theorem \ref{theorem:MainResult} for an arbitrary knotoid}
\label{sec:ProofForKnotoid}

Given a knotoid $K$
and its minimal diagram $D$, i.e., $\crn(D) =\crn(K)$.
Consider a \fkd~$F$ which is obtained from $D$ as a result of forgetting over/under-crossing information in all crossings of $D$.
By Theorem~\ref{Theorem:NonPrimeFlatKnotoids}:
$$\crn(F) \geq 2 \hh(F).$$
Since, by definition, the height of a knotoid is the minimum of the height over all representative diagram
$\hh(F)\geq \hh(K)$,
hence
$$\crn(K) = \crn(F) \geq 2 \hh(F) \geq 2 \hh(K).$$
This completes the proof of Theorem~\ref{theorem:MainResult}.

\section{Proof of~Corollary~\ref{corol:Bridge}}
\label{sec:ProofOfCorollary}

\input{Bridge}

Let $B$ be a bridge of the length $k(D)$.
Pick two points in the diagram:
$v$ placing just before the first crossing in the bridge $B$
and $u$ placing just after the last crossing in $B$ (see Figure \ref{Figure:Bridge}).
Without loss of generality we can think
that the bridge $B$ starts at $v$ and ends at $u$.
Then $B$ is a simple arc passing through $k(D)$ crossings of the diagram.
Observe that we can regard $D \setminus B$
as a knotoid diagram starting at $u$ and ending at $v$.
Denote by $F$ corresponding \fkd,
i.e., the result of forgetting of over/under-data in the diagram.
The number of crossings $F$ is equal to $\crn(D) -k(D)$.
The arc $[u,v]$ (we mean $B$ with reversed orientation)
is a shortcut of $F$ intersecting $F$ in $k(D)$ points.
By hypothesis, the diagram $D$ is minimal,
hence the shortcut is minimal also.
Hence $\hh(F) =k(D)$.
Using~Theorem~\ref{Theorem:NonPrimeFlatKnotoids}
and two equalities above we have
$$\crn(F) =\crn(D) -k(D) \geq 2 \hh(F) =2 k(D)$$
thus $\crn(D) \geq 3 k(D)$.
This completes the proof.

\section*{Acknowledgments}

The work was supported by RFBR (grant number 20-01-00127).

\end{document}

%% file: Reidemeister.tex
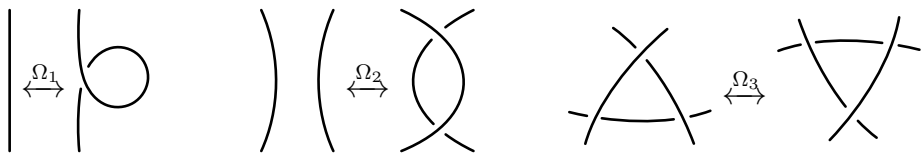
\begin{figure}[h]
	\begin{center}
		\begin{minipage}{0.25\textwidth}
		\begin{center}
			\begin{tikzpicture}[scale=0.35, yscale=-1.000000, baseline=-8.3ex]
				\path[draw=black, line cap=round, line width=\figuresLineThickness] (1.7325,1.0068) -- (1.7325,6.3635);
			\end{tikzpicture}
				$\xleftrightarrow{\Omega_1}$
				\begin{tikzpicture}[scale=0.35, yscale=-1.000000, baseline=-8.3ex]
				\path[shift={(8.21854,0)},draw=black,line cap=round, line width=\figuresLineThickness] (1.7325,6.3635) .. controls (1.6743,5.7347) and (1.6743,5.1006) .. (1.7325,4.4719) .. controls (1.7471,4.3139) and (1.7658,4.1552) .. (1.7934,3.9985);
				\path[shift={(8.21854,0)},draw=black,line cap=round, line width=\figuresLineThickness] (2.0793,3.1527) .. controls (2.1704,2.9937) and (2.2846,2.8476) .. (2.4227,2.7274) .. controls (2.5840,2.5870) and (2.7785,2.4827) .. (2.9877,2.4388) .. controls (3.3353,2.3659) and (3.7151,2.4706) .. (3.9772,2.7102) .. controls (4.2393,2.9498) and (4.3776,3.3179) .. (4.3395,3.6710) .. controls (4.3015,4.0241) and (4.0887,4.3537) .. (3.7833,4.5350) .. controls (3.4780,4.7164) and (3.0874,4.7456) .. (2.7579,4.6133) .. controls (2.5771,4.5407) and (2.4160,4.4229) .. (2.2841,4.2796) .. controls (2.1522,4.1363) and (2.0490,3.9679) .. (1.9713,3.7893) .. controls (1.8159,3.4321) and (1.7627,3.0393) .. (1.7325,2.6509) .. controls (1.6900,2.1040) and (1.6900,1.5538) .. (1.7325,1.0068);
			\end{tikzpicture}
		\end{center}
		\end{minipage}
		\begin{minipage}{0.3\textwidth}
		\begin{center}
			\begin{tikzpicture}[scale=0.35, yscale=-1.000000, baseline=-8.3ex]
				\path[draw=black, line cap=round, line width=\figuresLineThickness] (8.5742,6.3635) .. controls (9.2579,4.8696) and (9.3874,2.8454) .. (8.5742,1.0068);
				\path[draw=black,line cap=round, line width=\figuresLineThickness] (11.3110,6.3635) .. controls (10.6273,4.8696) and (10.4978,2.8454) .. (11.3110,1.0068);
			\end{tikzpicture}
		 			$\xleftrightarrow{\Omega_2}$
		 			\begin{tikzpicture}[scale=0.35, yscale=-1.000000, baseline=-8.3ex]
				\path[draw=black,line cap=round, line width=\figuresLineThickness] (14.3375,6.3635) .. controls (17.5139,5.0641) and (17.4666,2.4388) .. (14.3375,1.0068);
				\path[draw=black,line cap=round, line width=\figuresLineThickness] (17.0743,6.3635) .. controls (16.6679,6.1669) and (16.3161,5.9553) .. (16.0185,5.7326);
				\path[draw=black,line cap=round, line width=\figuresLineThickness] (15.5553,5.3349) .. controls (14.5444,4.3272) and (14.5314,3.1568) .. (15.4808,2.1364);
				\path[draw=black,line cap=round, line width=\figuresLineThickness] (16.0092,1.6603) .. controls (16.3097,1.4296) and (16.6649,1.2104) .. (17.0743,1.0068);
			\end{tikzpicture}
		\end{center}
		\end{minipage}
		\begin{minipage}{0.4\textwidth}
		\begin{center}
			\begin{tikzpicture}[scale=0.35, yscale=-1.000000, baseline=-8.3ex]
				\path[draw=black,line cap=round, line width=\figuresLineThickness] (20.4607,5.8164) .. controls (20.8809,4.3245) and (22.1508,2.6571) .. (23.5724,1.4563);
				\path[draw=black,line cap=round, line width=\figuresLineThickness] (25.2101,4.5622) .. controls (24.9753,4.6492) and (24.7238,4.7209) .. (24.4594,4.7782);
				\path[draw=black,line cap=round, line width=\figuresLineThickness] (23.8183,4.8863) .. controls (22.9619,4.9932) and (22.0155,4.9726) .. (21.0887,4.8464);
				\path[draw=black,line cap=round, line width=\figuresLineThickness] (20.5227,4.7558) .. controls (20.2966,4.7141) and (20.0730,4.6662) .. (19.8536,4.6125);
				\path[draw=black,line cap=round, line width=\figuresLineThickness] (21.5208,1.4158) .. controls (21.8082,1.6195) and (22.0894,1.8677) .. (22.3600,2.1500);
				\path[draw=black,line cap=round, line width=\figuresLineThickness] (22.7971,2.6460) .. controls (23.5313,3.5485) and (24.1605,4.6822) .. (24.5840,5.8102);
			\end{tikzpicture}
					$\xleftrightarrow{\Omega_3}$
		 			\begin{tikzpicture}[scale=0.35, yscale=-1.000000, baseline=-8.3ex]
				\path[draw=black,line cap=round, line width=\figuresLineThickness] (30.0930,5.5851) .. controls (29.8497,5.4026) and (29.6119,5.1872) .. (29.3823,4.9456);
				\path[draw=black,line cap=round, line width=\figuresLineThickness] (28.9133,4.4024) .. controls (28.1758,3.4675) and (27.5531,2.2831) .. (27.1482,1.1105);
				\path[draw=black,line cap=round, line width=\figuresLineThickness] (30.9484,1.0097) .. controls (30.6879,2.5376) and (29.6009,4.3296) .. (28.3138,5.6736);
				\path[draw=black,line cap=round, line width=\figuresLineThickness] (26.3630,2.2824) .. controls (26.6251,2.1859) and (26.9079,2.1083) .. (27.2061,2.0487);
				\path[draw=black,line cap=round, line width=\figuresLineThickness] (27.7906,1.9565) .. controls (28.5821,1.8626) and (29.4486,1.8769) .. (30.3044,1.9819);
				\path[draw=black,line cap=round, line width=\figuresLineThickness] (30.9322,2.0756) .. controls (31.1987,2.1225) and (31.4621,2.1780) .. (31.7196,2.2415);
			\end{tikzpicture}
		\end{center}
		\end{minipage}
		\caption{\label{Figure:Reidemeister}Three Reidemeister moves $\Omega_1, \Omega_2, \Omega_3$}
	\end{center}
\end{figure}

%% file: NonPrime.tex
\begin{figure}
	\begin{center}
		\begin{minipage}{0.49\textwidth}
			\begin{center}
				\begin{tikzpicture}[scale=0.3, yscale=-1.0]
					\path[fill=black] (0.3846,8.8002) circle (\figuresPointSize);
					\path[fill=black] (5.4323,8.8640) circle (\figuresPointSize);
					\path[draw=black,fill=lightgray,line cap=round, line width=\figuresThinLineThickness] (10.6227,4.7535) .. controls (10.6227,6.6121) and (9.1160,8.1188) .. (7.2574,8.1188) .. controls (5.3987,8.1188) and (3.8920,6.6121) .. (3.8920,4.7535) .. controls (3.8920,2.8948) and (5.3987,1.3881) .. (7.2574,1.3881) .. controls (9.1160,1.3881) and (10.6227,2.8948) .. (10.6227,4.7535) -- cycle;
					\path[fill=black] (5.0562,7.2992) circle (\figuresSmallPointSize);
					\path[fill=black] (8.1339,8.0027) circle (\figuresSmallPointSize);
					\path[draw=black, line cap=round, line width=\figuresLineThickness] (0.3846,8.8002) .. controls (1.3398,9.0485) and (2.3752,8.9772) .. (3.2874,8.6004) .. controls (4.1996,8.2236) and (4.9834,7.5433) .. (5.4848,6.6932) .. controls (5.7823,6.1890) and (5.9803,5.6330) .. (6.2439,5.1103) .. controls (6.5075,4.5876) and (6.8540,4.0820) .. (7.3577,3.7836) .. controls (7.6431,3.6145) and (7.9726,3.5187) .. (8.3043,3.5157) .. controls (8.6360,3.5128) and (8.9690,3.6038) .. (9.2493,3.7813) .. controls (9.5295,3.9589) and (9.7551,4.2232) .. (9.8797,4.5307) .. controls (10.0042,4.8381) and (10.0258,5.1875) .. (9.9329,5.5060) .. controls (9.8415,5.8191) and (9.6422,6.0968) .. (9.3858,6.2986) .. controls (9.1295,6.5004) and (8.8182,6.6278) .. (8.4966,6.6827) .. controls (7.8535,6.7924) and (7.1897,6.6153) .. (6.6052,6.3253) .. controls (6.2118,6.1302) and (5.8414,5.8816) .. (5.5383,5.5638) .. controls (5.2352,5.2459) and (5.0012,4.8562) .. (4.8992,4.4291) .. controls (4.7704,3.8899) and (4.8668,3.2913) .. (5.1969,2.8460) .. controls (5.3620,2.6233) and (5.5824,2.4410) .. (5.8348,2.3264) .. controls (6.0872,2.2118) and (6.3712,2.1658) .. (6.6461,2.2010) .. controls (6.9859,2.2445) and (7.3042,2.4108) .. (7.5544,2.6448) .. controls (7.8045,2.8787) and (7.9887,3.1778) .. (8.1129,3.4970) .. controls (8.3614,4.1353) and (8.3729,4.8377) .. (8.3777,5.5227) .. controls (8.3825,6.2034) and (8.3834,6.8897) .. (8.2480,7.5569) .. controls (8.1127,8.2240) and (7.8301,8.8786) .. (7.3410,9.3520) .. controls (6.9290,9.7509) and (6.3830,10.0038) .. (5.8149,10.0824) .. controls (5.2469,10.1609) and (4.6593,10.0679) .. (4.1339,9.8381) .. controls (3.6085,9.6084) and (3.1453,9.2445) .. (2.7799,8.8026) .. controls (2.4144,8.3607) and (2.1457,7.8421) .. (1.9732,7.2952) .. controls (1.7050,6.4446) and (1.6702,5.5217) .. (1.8729,4.6531) .. controls (2.0656,3.8269) and (2.4716,3.0524) .. (3.0349,2.4181) .. controls (3.5983,1.7837) and (4.3175,1.2901) .. (5.1089,0.9844) .. controls (6.6918,0.3731) and (8.5442,0.5489) .. (10.0165,1.3923) .. controls (10.7197,1.7952) and (11.3420,2.3443) .. (11.8070,3.0080) .. controls (12.2721,3.6717) and (12.5772,4.4506) .. (12.6595,5.2568) .. controls (12.7419,6.0631) and (12.5980,6.8949) .. (12.2256,7.6146) .. controls (11.8531,8.3343) and (11.2500,8.9371) .. (10.5182,9.2851) .. controls (9.9673,9.5472) and (9.3552,9.6635) .. (8.7456,9.6865) .. controls (7.5934,9.7300) and (6.4304,9.4413) .. (5.4323,8.8640);
				\end{tikzpicture}
			\end{center}

		\end{minipage}
		\hfill
		\begin{minipage}{0.49\textwidth}
			\begin{center}
				\begin{tikzpicture}[scale=0.3, yscale=-1.0]
					\path[draw=black,fill=lightgray,line cap=round, line width=\figuresThinLineThickness] (21.2806,6.5081) circle (3.0);
					\path[fill=black] (17.2739,11.4088) circle (\figuresPointSize);
					\path[fill=black] (22.6040,7.2595) circle (\figuresPointSize);
					\path[fill=black] (22.8479,3.9647) circle (\figuresSmallPointSize);
					\path[draw=black,line cap=round,line width=\figuresLineThickness] (17.2739,11.4088) .. controls (16.4331,10.5073) and (15.9423,9.2875) .. (15.9244,8.0549) .. controls (15.9066,6.8223) and (16.3617,5.5888) .. (17.1760,4.6632) .. controls (17.9902,3.7377) and (19.1557,3.1290) .. (20.3805,2.9897) .. controls (21.6053,2.8503) and (22.8778,3.1817) .. (23.8791,3.9007) .. controls (24.4949,4.3428) and (25.0125,4.9328) .. (25.3203,5.6256) .. controls (25.6281,6.3185) and (25.7176,7.1157) .. (25.5178,7.8471) .. controls (25.2991,8.6479) and (24.7384,9.3396) .. (24.0253,9.7645) .. controls (23.3122,10.1895) and (22.4566,10.3540) .. (21.6301,10.2762) .. controls (20.8036,10.1984) and (20.0071,9.8854) .. (19.3171,9.4238) .. controls (18.6271,8.9622) and (18.0410,8.3551) .. (17.5581,7.6798) .. controls (16.9137,6.7787) and (16.4426,5.7310) .. (16.3476,4.6272) .. controls (16.2526,3.5234) and (16.5620,2.3625) .. (17.3073,1.5428) .. controls (17.6801,1.1328) and (18.1543,0.8149) .. (18.6770,0.6308) .. controls (19.1997,0.4467) and (19.7699,0.3974) .. (20.3155,0.4941) .. controls (20.8612,0.5909) and (21.3808,0.8342) .. (21.8012,1.1952) .. controls (22.2215,1.5563) and (22.5410,2.0347) .. (22.7085,2.5629) .. controls (22.8688,3.0685) and (22.8899,3.6105) .. (22.8317,4.1377) .. controls (22.7735,4.6649) and (22.6383,5.1803) .. (22.4911,5.6899) .. controls (22.3000,6.3518) and (22.0771,7.0278) .. (21.6254,7.5481) .. controls (21.3996,7.8082) and (21.1174,8.0239) .. (20.7951,8.1454) .. controls (20.4727,8.2669) and (20.1091,8.2900) .. (19.7822,8.1815) .. controls (19.5496,8.1043) and (19.3398,7.9623) .. (19.1773,7.7788) .. controls (19.0148,7.5954) and (18.8996,7.3714) .. (18.8392,7.1339) .. controls (18.7184,6.6590) and (18.8222,6.1402) .. (19.0798,5.7234) .. controls (19.3623,5.2663) and (19.8217,4.9285) .. (20.3328,4.7631) .. controls (20.8440,4.5976) and (21.4030,4.5999) .. (21.9226,4.7368) .. controls (22.2777,4.8303) and (22.6179,4.9864) .. (22.9101,5.2088) .. controls (23.2024,5.4312) and (23.4456,5.7209) .. (23.5986,6.0547) .. controls (23.7516,6.3885) and (23.8123,6.7667) .. (23.7544,7.1293) .. controls (23.6966,7.4920) and (23.5179,7.8370) .. (23.2436,8.0812) .. controls (23.0182,8.2819) and (22.7336,8.4114) .. (22.4372,8.4687) .. controls (22.1408,8.5259) and (21.8331,8.5126) .. (21.5380,8.4490) .. controls (21.1797,8.3719) and (20.8335,8.2179) .. (20.5599,7.9740) .. controls (20.2864,7.7300) and (20.0900,7.3920) .. (20.0497,7.0277) .. controls (20.0072,6.6436) and (20.1455,6.2434) .. (20.4160,5.9675) .. controls (20.6865,5.6916) and (21.0839,5.5454) .. (21.4688,5.5803) .. controls (21.8536,5.6153) and (22.2182,5.8306) .. (22.4346,6.1507) .. controls (22.6510,6.4708) and (22.7150,6.8894) .. (22.6040,7.2595);
				\end{tikzpicture}
			\end{center}
		\end{minipage}
		\caption{\label{Figure:NonPrime}Two examples of non-prime FKD. The one on the left-hand side does not satisfies the condition~(I), the one on the right-hand side does not satisfies the condition~(II)}
	\end{center}
\end{figure}
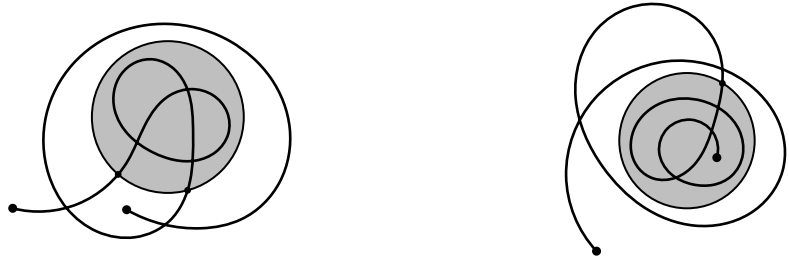

%% file: Gamma.tex
\begin{figure}[h]
	\begin{center}
		\begin{tikzpicture}[scale=0.3, yscale=-1.0]
			\path[fill=lightgray] (4.5753,4.8142) .. controls (4.5753,4.8142) and (3.6752,4.5767) .. (2.8002,5.2142) .. controls (1.9251,5.8517) and (1.5947,6.7227) .. (1.6376,7.4518) .. controls (1.6805,8.1810) and (1.9018,9.0565) .. (2.4547,9.7483) .. controls (3.0076,10.4402) and (4.2239,11.3367) .. (5.3753,11.6146) .. controls (7.1879,12.0521) and (8.7005,12.2771) .. (10.7132,11.7646) .. controls (12.7258,11.2521) and (14.0384,10.1520) .. (14.5634,9.2770) .. controls (15.0884,8.4019) and (15.3635,7.2018) .. (15.1259,6.2643) .. controls (14.8884,5.3267) and (14.5634,4.8142) .. (14.5634,4.8142) .. controls (14.5634,4.8142) and (14.4513,6.0658) .. (13.8759,6.8268) .. controls (13.3005,7.5878) and (12.5383,8.2894) .. (11.6507,8.5144) .. controls (10.7632,8.7394) and (10.1131,8.5144) .. (10.1131,8.5144) .. controls (10.9166,7.5515) and (11.1353,7.0855) .. (11.2257,6.4268) .. controls (11.3161,5.7681) and (11.2206,4.8807) .. (10.1631,4.1641) .. controls (9.1056,3.4476) and (8.0255,4.0266) .. (7.2379,4.5142) .. controls (6.4504,5.0017) and (6.0129,5.6892) .. (5.5628,6.2393) .. controls (5.1128,6.7893) and (4.8158,7.2372) .. (4.8158,7.2372) .. controls (4.8158,7.2372) and (4.5339,6.7525) .. (4.4945,6.0333) .. controls (4.4550,5.3142) and (4.5753,4.8142) .. (4.5753,4.8142) -- cycle;
			\path[draw=black,line cap=round,line width=\figuresThinLineThickness] (5.8016,10.5256) .. controls (5.9299,10.1125) and (6.1475,9.7274) .. (6.4351,9.4042) .. controls (6.8347,8.9553) and (7.3633,8.6285) .. (7.7256,8.1490) .. controls (8.0601,7.7065) and (8.2317,7.1614) .. (8.5388,6.6994) .. controls (8.7069,6.4465) and (8.9150,6.2203) .. (9.1530,6.0317);
			\draw (6.9,8.0) node {$\gamma$};
			\path[fill=black] (5.8016,10.5256) circle (\figuresPointSize);
			\path[fill=black] (9.1530,6.0317) circle (\figuresPointSize);
			\path[draw=black,line cap=round,line width=\figuresLineThickness] (5.8016,10.5256) .. controls (5.3175,10.3052) and (4.9160,9.9086) .. (4.6896,9.4273) .. controls (4.4632,8.9460) and (4.4138,8.3838) .. (4.5527,7.8704) .. controls (4.7120,7.2816) and (5.0989,6.7849) .. (5.4785,6.3074) .. controls (6.0765,5.5554) and (6.6988,4.7898) .. (7.5305,4.3087) .. controls (7.9464,4.0682) and (8.4116,3.9043) .. (8.8910,3.8729) .. controls (9.3704,3.8416) and (9.8641,3.9479) .. (10.2672,4.2091) .. controls (10.6976,4.4880) and (11.0089,4.9349) .. (11.1545,5.4265) .. controls (11.3002,5.9182) and (11.2852,6.4511) .. (11.1457,6.9446) .. controls (10.9438,7.6589) and (10.4799,8.2944) .. (9.8666,8.7127) .. controls (9.2534,9.1310) and (8.4961,9.3308) .. (7.7555,9.2794) .. controls (7.0149,9.2281) and (6.2953,8.9276) .. (5.7315,8.4447) .. controls (5.1676,7.9619) and (4.7612,7.3002) .. (4.5757,6.5814) .. controls (4.3906,5.8641) and (4.4250,5.0916) .. (4.6733,4.3936) .. controls (4.9215,3.6957) and (5.3825,3.0750) .. (5.9790,2.6356) .. controls (6.7519,2.0663) and (7.7294,1.8129) .. (8.6893,1.8202) .. controls (9.6492,1.8275) and (10.5944,2.0826) .. (11.4803,2.4522) .. controls (12.3245,2.8044) and (13.1326,3.2675) .. (13.7920,3.9015) .. controls (14.4514,4.5355) and (14.9567,5.3501) .. (15.1288,6.2485) .. controls (15.3006,7.1445) and (15.1314,8.0922) .. (14.7126,8.9028) .. controls (14.2939,9.7135) and (13.6345,10.3889) .. (12.8655,10.8800) .. controls (11.3275,11.8621) and (9.4177,12.0944) .. (7.5964,11.9820) .. controls (6.6407,11.9230) and (5.6826,11.7742) .. (4.7887,11.4312) .. controls (3.8947,11.0882) and (3.0644,10.5416) .. (2.4800,9.7832) .. controls (2.1488,9.3533) and (1.8994,8.8577) .. (1.7729,8.3299) .. controls (1.6465,7.8021) and (1.6451,7.2421) .. (1.7900,6.7190) .. controls (1.9349,6.1960) and (2.2283,5.7119) .. (2.6395,5.3576) .. controls (3.0506,5.0033) and (3.5798,4.7831) .. (4.1221,4.7624) .. controls (4.6754,4.7413) and (5.2227,4.9244) .. (5.6990,5.2069) .. controls (6.1752,5.4894) and (6.5871,5.8679) .. (6.9742,6.2638) .. controls (7.3613,6.6597) and (7.7278,7.0763) .. (8.1345,7.4521) .. controls (8.5412,7.8279) and (8.9931,8.1650) .. (9.5067,8.3719) .. controls (10.0241,8.5803) and (10.5952,8.6503) .. (11.1491,8.5843) .. controls (11.7031,8.5183) and (12.2392,8.3176) .. (12.7067,8.0131) .. controls (13.6416,7.4041) and (14.2796,6.3859) .. (14.5058,5.2933) .. controls (14.6980,4.3650) and (14.6006,3.3677) .. (14.1679,2.5242) .. controls (13.7352,1.6808) and (12.9566,1.0066) .. (12.0400,0.7647) .. controls (11.3781,0.5900) and (10.6583,0.6432) .. (10.0293,0.9134) .. controls (9.4002,1.1836) and (8.8660,1.6689) .. (8.5370,2.2693) .. controls (8.2079,2.8696) and (8.0861,3.5810) .. (8.1967,4.2566) .. controls (8.3074,4.9322) and (8.6497,5.5676) .. (9.1530,6.0317);
			\path[draw=black,line cap=round,line width=\figuresThickLineThickness] (6.1264,5.5088) .. controls (6.4136,5.7456) and (6.6894,5.9963) .. (6.9525,6.2596) .. controls (7.3439,6.6513) and (7.7072,7.0710) .. (8.1129,7.4479) .. controls (8.5185,7.8249) and (8.9713,8.1610) .. (9.4850,8.3677) .. controls (9.6805,8.4463) and (9.8837,8.5057) .. (10.0908,8.5448)(10.0556,8.5517) .. controls (9.9875,8.6068) and (9.9173,8.6591) .. (9.8450,8.7085) .. controls (9.2321,9.1276) and (8.4744,9.3270) .. (7.7338,9.2752) .. controls (6.9933,9.2235) and (6.2740,8.9229) .. (5.7098,8.4405) .. controls (5.3253,8.1118) and (5.0117,7.7005) .. (4.7965,7.2427);
		\end{tikzpicture}
		\caption{\label{Figure:Gamma}FKD with a shortcut $\gamma$. All three $\gamma$-regions  are shaded, and both $\gamma$-edges are drawn with thick lines}
	\end{center}
\end{figure}
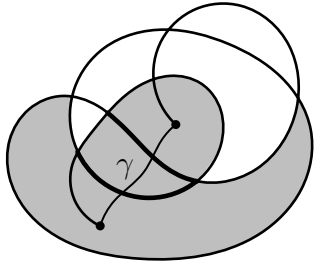

%% file: gammaReduction0102.tex
\begin{figure}[h]
	\begin{center}
		\begin{minipage}{0.45\textwidth}
			\begin{center}
				\begin{tikzpicture}[scale=0.3, yscale=-1.0]
					\path[draw=black,line cap=round,line width=\figuresThickLineThickness] (0.8282,6.8295) .. controls (2.7506,6.3950) and (6.6540,3.7084) .. (7.2923,2.0844);
					\path[draw=black,line cap=round,line width=\figuresThickLineThickness] (5.6181,1.9909) .. controls (6.6387,3.6769) and (10.4226,6.5294) .. (12.1652,6.6208);
					\path[draw=black,line cap=round,line width=\figuresThickLineThickness] (11.6243,4.9732) .. controls (10.3428,6.4704) and (8.8190,10.9574) .. (9.2779,12.6409);
					\path[draw=black,line cap=round,line width=\figuresThickLineThickness] (4.1352,12.8389) .. controls (4.2915,10.8743) and (2.8857,6.3490) .. (1.5246,5.2571);
					\path[draw=black,line cap=round,line width=\figuresThickLineThickness] (10.4906,11.5070) .. controls (8.6524,10.7962) and (3.9141,10.8511) .. (2.4790,11.8437);
					\path[fill=black] (4.0361,11.2807) circle (\figuresSmallPointSize);
					\path[fill=black] (2.3210,6.2548) circle (\figuresSmallPointSize);
					\path[fill=black] (6.5650,3.1536) circle (\figuresSmallPointSize);
					\path[fill=black] (10.8225,6.2255) circle (\figuresSmallPointSize);
					\path[fill=black] (9.2378,11.1963) circle (\figuresSmallPointSize);
					\path[draw=black,->,line cap=round,line width=\figuresLineThickness] (1.2929,9.7254) .. controls (1.9793,9.8049) and (2.6921,9.6318) .. (3.2656,9.2465) .. controls (3.8391,8.8612) and (4.2687,8.2667) .. (4.4545,7.6012) .. controls (4.5863,7.1289) and (4.5977,6.6284) .. (4.5337,6.1422) .. controls (4.4698,5.6560) and (4.3323,5.1822) .. (4.1705,4.7193) .. controls (4.0912,4.4922) and (4.0053,4.2641) .. (3.9299,4.0339);
					\path[draw=black,dash pattern=on 5.0 off 5.0,line cap=round,line width=\figuresLineThickness] (3.9299,4.0339) .. controls (3.7987,3.6336) and (3.6991,3.2270) .. (3.7200,2.8093) .. controls (3.7487,2.2330) and (4.0155,1.6783) .. (4.4232,1.2699) .. controls (4.8308,0.8615) and (5.3725,0.5969) .. (5.9395,0.4899) .. controls (6.5920,0.3668) and (7.2907,0.4530) .. (7.8702,0.7772) .. controls (8.4496,1.1015) and (8.8976,1.6715) .. (9.0268,2.3228) .. controls (9.1199,2.7916) and (9.0504,3.2783) .. (8.9312,3.7411);
					\path[draw=black,->,line cap=round,line width=\figuresLineThickness] (8.9312,3.7411) .. controls (8.8119,4.2039) and (8.6435,4.6533) .. (8.5342,5.1185) .. controls (8.3119,6.0647) and (8.3491,7.0972) .. (8.7909,7.9630) .. controls (9.0608,8.4918) and (9.4765,8.9453) .. (9.9800,9.2600) .. controls (10.4834,9.5747) and (11.0732,9.7497) .. (11.6668,9.7606);
					\path[fill=black] (3.5088,9.0635) circle (\figuresSmallPointSize) node[left] {$p_1$};
					\path[fill=black] (4.2901,5.0801) circle (\figuresSmallPointSize) node[left] {$p_2$};
					\path[fill=black] (8.5941,4.8843) circle (\figuresSmallPointSize) node[right] {$p_{n-1}$};
					\path[fill=black] (9.6983,9.0624) circle (\figuresSmallPointSize) node[right] {$p_n$};
					
					\path[draw=black,line cap=round,line width=\figuresThinLineThickness] (3.5088,9.0635) .. controls (5.6885,9.8931) and (7.0145,10.0448) .. (9.6983,9.0625);
					\draw (6.4, -0.5) node {$\gamma$};
					\draw (6.6, 7.6) node {$\Delta$};
				\end{tikzpicture}
			\end{center}
		\end{minipage}
		\hfill
		\begin{minipage}{0.49\textwidth}
			\begin{center}
				\begin{tikzpicture}[scale=0.35, yscale=-1.0]
					\path[draw=black,line cap=round,line width=\figuresThickLineThickness] (1.4214,8.9841) .. controls (3.7959,7.1447) and (4.4313,4.1180) .. (3.4447,2.4458);
					\path[draw=black,line cap=round,line width=\figuresThickLineThickness] (13.9431,8.9841) .. controls (11.5685,7.1447) and (10.9331,4.1180) .. (11.9197,2.4458);
					\path[draw=black,line cap=round,line width=\figuresThickLineThickness] (0.7190,6.8437) .. controls (4.8661,4.7535) and (11.0533,7.1447) .. (13.8626,4.0679);
					\path[draw=black,->,line cap=round,line width=\figuresLineThickness] (3.4949,8.9841) .. controls (4.4539,7.8954) and (5.1726,6.5960) .. (5.5852,5.2050) .. controls (5.7900,4.5144) and (5.9249,3.7907) .. (6.2860,3.1674);
					\path[draw=black,->,line cap=round,line width=\figuresLineThickness] (6.2860,3.1674) .. controls (6.4665,2.8557) and (6.7045,2.5726) .. (7.0041,2.3727) .. controls (7.3038,2.1728) and (7.6680,2.0604) .. (8.0266,2.0947) .. controls (8.3169,2.1224) and (8.5941,2.2448) .. (8.8261,2.4215) .. controls (9.0581,2.5982) and (9.2465,2.8278) .. (9.3971,3.0776) .. controls (9.6981,3.5772) and (9.8483,4.1501) .. (10.0332,4.7033) .. controls (10.5475,6.2417) and (11.3582,7.6804) .. (12.4077,8.9173);
					\path[fill=black] (3.5766,6.0242) circle (\figuresSmallPointSize);
					\path[fill=black] (11.6073,5.4015) circle (\figuresSmallPointSize);
					\path[fill=black] (5.3590,5.8782) circle (\figuresSmallPointSize) node[above right] {$p_1$};
					\path[fill=black] (10.3947,5.6555) circle (\figuresSmallPointSize) node[above left] {$p_2$};
					\draw (7.8, 1.5) node {$\gamma$};
					\draw (7.8, 4.0) node {$\Delta$};
				\end{tikzpicture}
			\end{center}
		\end{minipage}
		\caption{\label{Figure:gammaReduction0102}Reduce the number of intersections of the shortcut $\gamma$ with edges of the FKD}
	\end{center}
\end{figure}
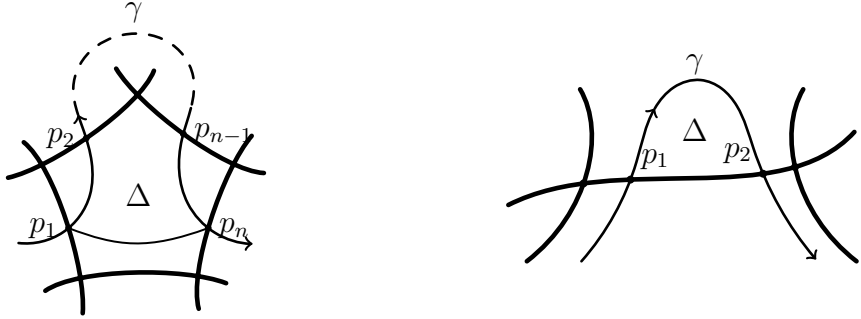

%% file: VertexTypeExample.tex
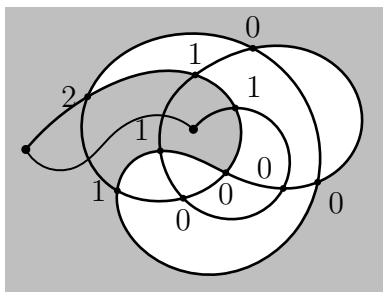
\begin{figure}[h]
	\begin{center}
		\begin{tikzpicture}[scale=0.3, yscale=-1.0]
			\path[fill=lightgray,line join=round,line cap=round,miter limit=4.00,line width=0.148pt] (1.9988,0.1740) -- (18.7319,0.1740) -- (18.7319,12.7818) -- (1.9988,12.7818) -- cycle;
			\path[fill=white, line width=0.0] (5.6629,4.1391) .. controls (5.6629,4.1391) and (5.8707,3.5535) .. (6.4335,2.9351) .. controls (6.7550,2.5818) and (6.9704,2.3302) .. (7.9380,1.8390) .. controls (8.9056,1.3478) and (10.9630,1.3394) .. (11.6007,1.5015) .. controls (12.3315,1.6872) and (13.0258,2.0140) .. (13.0258,2.0140) .. controls (13.0258,2.0140) and (14.5009,1.7390) .. (15.4135,2.0765) .. controls (16.3260,2.4140) and (17.4136,3.2516) .. (17.7386,4.5142) .. controls (18.0636,5.7767) and (17.5011,6.4143) .. (17.3136,6.7768) .. controls (17.1261,7.1393) and (16.7114,7.4752) .. (15.8510,7.9894) .. controls (15.6260,8.7644) and (14.9459,10.1261) .. (14.3948,10.6770) .. controls (13.7259,11.3458) and (12.2107,12.1151) .. (11.0320,12.0418) .. controls (9.8534,11.9684) and (8.3321,11.3781) .. (7.6130,10.4270) .. controls (6.8504,9.4186) and (6.8113,8.3002) .. (7.2563,7.5152) .. controls (7.6536,6.8143) and (8.1197,6.5341) .. (9.0935,6.5695) .. controls (10.3348,6.6145) and (11.8132,7.5393) .. (11.8132,7.5393) .. controls (11.8132,7.5393) and (12.3477,6.7680) .. (12.4477,6.1804) .. controls (12.5477,5.5929) and (12.2309,4.4172) .. (11.7507,3.9891) .. controls (11.3062,3.5929) and (10.8625,3.2340) .. (10.0163,3.0750) .. controls (9.1700,2.9161) and (8.8594,3.0016) .. (7.7718,3.2391) .. controls (6.6877,3.4758) and (5.6629,4.1391) .. (5.6629,4.1391) -- cycle;
			\path[draw=black,line cap=round,line width=\figuresLineThickness] (2.9341,6.4976) .. controls (3.7147,5.5594) and (4.6587,4.7579) .. (5.7095,4.1372);
			\path[draw=black,line cap=round,line width=\figuresLineThickness] (5.7095,4.1372) .. controls (5.9721,3.9821) and (6.2415,3.8382) .. (6.5165,3.7064) .. controls (7.5772,3.1980) and (8.7698,2.8640) .. (9.9286,3.0655) .. controls (10.5080,3.1663) and (11.0683,3.4033) .. (11.5191,3.7810) .. controls (11.7994,4.0159) and (12.0355,4.3065) .. (12.2026,4.6312);
			\path[draw=black,line cap=round,line width=\figuresLineThickness] (12.2026,4.6312) .. controls (12.3042,4.8286) and (12.3804,5.0387) .. (12.4255,5.2564) .. controls (12.5141,5.6841) and (12.4824,6.1343) .. (12.3483,6.5499) .. controls (12.2142,6.9656) and (11.9792,7.3467) .. (11.6782,7.6632) .. controls (11.0763,8.2963) and (10.2237,8.6579) .. (9.3562,8.7613) .. controls (8.3859,8.8771) and (7.3522,8.6679) .. (6.5898,8.0566) .. controls (6.1172,7.6778) and (5.7634,7.1559) .. (5.5731,6.5809) .. controls (5.3829,6.0059) and (5.3551,5.3799) .. (5.4799,4.7872) .. controls (5.6538,3.9618) and (6.1212,3.2097) .. (6.7496,2.6470) .. controls (7.3780,2.0843) and (8.1615,1.7069) .. (8.9823,1.5124) .. controls (10.0748,1.2536) and (11.2427,1.3164) .. (12.2979,1.6999) .. controls (13.3531,2.0834) and (14.2908,2.7874) .. (14.9459,3.6992) .. controls (15.6010,4.6109) and (15.9693,5.7273) .. (15.9760,6.8500) .. controls (15.9782,7.2280) and (15.9393,7.6061) .. (15.8615,7.9760);
			\path[draw=black,line cap=round,line width=\figuresLineThickness] (15.8615,7.9760) .. controls (15.7081,8.7047) and (15.4037,9.4018) .. (14.9656,10.0039) .. controls (14.1605,11.1104) and (12.8915,11.8831) .. (11.5307,12.0270) .. controls (10.1699,12.1708) and (8.7418,11.6629) .. (7.8229,10.6491) .. controls (7.3484,10.1257) and (7.0063,9.4601) .. (6.9683,8.7548) .. controls (6.9493,8.4021) and (7.0072,8.0438) .. (7.1498,7.7207) .. controls (7.2924,7.3976) and (7.5208,7.1107) .. (7.8122,6.9111) .. controls (8.1189,6.7011) and (8.4879,6.5920) .. (8.8592,6.5679);
			\path[draw=black,line cap=round,line width=\figuresLineThickness] (8.8592,6.5679) .. controls (8.8771,6.5667) and (8.8950,6.5658) .. (8.9129,6.5650) .. controls (9.3022,6.5483) and (9.6915,6.6209) .. (10.0623,6.7406) .. controls (10.8039,6.9799) and (11.4714,7.4022) .. (12.1848,7.7157) .. controls (13.0108,8.0786) and (13.9185,8.2960) .. (14.8163,8.2062) .. controls (15.5098,8.1369) and (16.1892,7.8796) .. (16.7299,7.4399) .. controls (17.2706,7.0001) and (17.6659,6.3746) .. (17.7921,5.6891) .. controls (17.9380,4.8969) and (17.7158,4.0547) .. (17.2417,3.4034) .. controls (16.7675,2.7521) and (16.0552,2.2896) .. (15.2812,2.0661) .. controls (14.5072,1.8427) and (13.6754,1.8510) .. (12.8907,2.0335) .. controls (12.1060,2.2160) and (11.3668,2.5683) .. (10.6962,3.0146) .. controls (10.1349,3.3881) and (9.6080,3.8404) .. (9.2718,4.4247) .. controls (8.8233,5.2041) and (8.7555,6.1704) .. (8.9920,7.0379) .. controls (9.2065,7.8244) and (9.6711,8.5511) .. (10.3341,9.0255) .. controls (10.9970,9.4998) and (11.8607,9.7052) .. (12.6545,9.5198) .. controls (13.1100,9.4134) and (13.5360,9.1819) .. (13.8710,8.8554) .. controls (14.2059,8.5290) and (14.4485,8.1081) .. (14.5601,7.6539) .. controls (14.6716,7.1996) and (14.6514,6.7134) .. (14.4995,6.2710) .. controls (14.3477,5.8286) and (14.0643,5.4318) .. (13.6931,5.1471) .. controls (13.1984,4.7676) and (12.5494,4.5956) .. (11.9316,4.6801) .. controls (11.3138,4.7647) and (10.7349,5.1049) .. (10.3604,5.6034);
			\path[draw=black,line cap=round,line width=\figuresThinLineThickness] (2.9341,6.4976) .. controls (3.1272,6.8649) and (3.4572,7.1582) .. (3.8446,7.3069) .. controls (4.2321,7.4555) and (4.6736,7.4582) .. (5.0628,7.3143) .. controls (5.4402,7.1748) and (5.7570,6.9083) .. (6.0383,6.6207) .. controls (6.3197,6.3331) and (6.5746,6.0189) .. (6.8754,5.7517) .. controls (7.1350,5.5211) and (7.4320,5.3271) .. (7.7512,5.1908);
			\path[draw=black,line cap=round,line width=\figuresThinLineThickness] (7.7512,5.1908) .. controls (8.0456,5.0652) and (8.3588,4.9888) .. (8.6789,4.9783) .. controls (9.2918,4.9582) and (9.9094,5.1878) .. (10.3604,5.6034);
			\path[fill=black] (2.9341,6.4976) circle (\figuresPointSize);
			\path[fill=black] (10.3604,5.6034) circle (\figuresPointSize);
			\path[fill=black] (5.6735,4.1586) circle (\figuresSmallPointSize) node[left] {2};
			\path[fill=black] (12.2168,4.6592) circle (\figuresSmallPointSize) node[above right] {1};
			\path[fill=black] (15.8675,7.9471) circle (\figuresSmallPointSize) node[below right] {0};
			\path[fill=black] (8.8937,6.5573) circle (\figuresSmallPointSize) node[above left] {1};
			\path[fill=black] (12.9974,2.0098) circle (\figuresSmallPointSize) node[above] {0};
			\path[fill=black] (10.4405,3.1925) circle (\figuresSmallPointSize) node[above] {1};
			\path[fill=black] (11.7979,7.5299) circle (\figuresSmallPointSize) node[below] {0};
			\path[fill=black] (14.3415,8.2260) circle (\figuresSmallPointSize) node[above left] {0};
			\path[fill=black] (9.9079,8.6596) circle (\figuresSmallPointSize) node[below] {0};
			\path[fill=black] (6.9892,8.3275) circle (\figuresSmallPointSize) node[left] {1};
		\end{tikzpicture}
		\caption{\label{Figure:VertexTypeExample}FKD and types of all it crossings}
	\end{center}
\end{figure}

%% file: VertexTypeNeightborhoods.tex
\begin{figure}[h]
	\begin{center}
		\begin{minipage}{0.12\textwidth}
			\begin{center}
				\begin{tikzpicture}[scale=0.4, yscale=-1.0]
					\path[draw=black,line cap=round,line width=\figuresLineThickness] (2.8640,1.2366) -- (2.8640,4.8078);
					\path[draw=black,line cap=round,line width=\figuresLineThickness] (4.6495,3.0222) -- (1.0784,3.0222);
					\path[fill=black] (2.8640,3.0222) circle (\figuresSmallPointSize) node[below left] {0};
				\end{tikzpicture}
			\end{center}
		\end{minipage}
		\hfill
		\begin{minipage}{0.12\textwidth}
			\begin{center}
				\begin{tikzpicture}[scale=0.4, yscale=-1.0]
					\path[scale=-1.000,fill=lightgray] (-9.1943,-7.0975)arc(180.000:270.000:1.794) -- (-7.4005,-7.0975) -- cycle;
					\path[draw=black,line cap=round,line width=\figuresLineThickness] (7.4005,5.3119) -- (7.4005,8.8830);
					\path[draw=black,line cap=round,line width=\figuresLineThickness] (9.1860,7.0975) -- (5.6149,7.0975);
					\path[fill=black] (7.4005,7.0975) circle (\figuresSmallPointSize) node[below left] {0};
				\end{tikzpicture}
			\end{center}
		\end{minipage}
		\hfill
		\begin{minipage}{0.12\textwidth}
			\begin{center}
				\begin{tikzpicture}[scale=0.4, yscale=-1.0]
					\path[scale=-1.000,fill=lightgray] (-9.1943,-3.0222)arc(180.000:270.000:1.794) -- (-7.4005,-3.0222) -- cycle;
					\path[fill=lightgray] (5.6066,3.0222)arc(180.000:270.000:1.794) -- (7.4005,3.0222) -- cycle;
					\path[draw=black,line cap=round,line width=\figuresLineThickness] (7.4005,1.2367) -- (7.4005,4.8078);
					\path[draw=black,line cap=round,line width=\figuresLineThickness] (9.1860,3.0222) -- (5.6149,3.0222);
					\path[fill=black] (7.4005,3.0222) circle (\figuresSmallPointSize) node[below left] {0};
				\end{tikzpicture}
			\end{center}
		\end{minipage}
		\hfill
		\begin{minipage}{0.12\textwidth}
			\begin{center}
				\begin{tikzpicture}[scale=0.4, yscale=-1.0]
					\path[scale=-1.000,fill=lightgray] (-13.2339,-3.0222)arc(180.000:270.000:1.794) -- (-11.4401,-3.0222) -- cycle;
					\path[rotate=90.0,fill=lightgray] (1.2283,-11.4401)arc(180.000:270.000:1.794) -- (3.0222,-11.4401) -- cycle;
					\path[draw=black,line cap=round,line width=\figuresLineThickness] (11.4401,1.2367) -- (11.4401,4.8078);
					\path[draw=black,line cap=round,line width=\figuresThickLineThickness] (13.2256,3.0222) -- (11.4507,3.0222);
					\path[draw=black,line cap=round,line width=\figuresLineThickness] (11.4507,3.0222) -- (9.6545,3.0222);
					\path[fill=black] (11.4401,3.0222) circle (\figuresSmallPointSize) node[below left] {1};
				\end{tikzpicture}
			\end{center}
		\end{minipage}
		\hfill
		\begin{minipage}{0.12\textwidth}
			\begin{center}
				\begin{tikzpicture}[scale=0.4, yscale=-1.0]
					\path[fill=lightgray] (13.7965,3.0222)arc(180.000:270.000:1.794) -- (15.5903,3.0222) -- cycle;
					\path[scale=-1.000,fill=lightgray] (-17.3735,-3.0222)arc(180.000:270.000:1.794) -- (-15.5797,-3.0222) -- cycle;
					\path[rotate=90.0,fill=lightgray] (1.2283,-15.5797)arc(180.000:270.000:1.794) -- (3.0222,-15.5797) -- cycle;
					\path[draw=black,line cap=round,line width=\figuresThickLineThickness] (15.5797,1.2366) -- (15.5797,3.0141);
					\path[draw=black,line cap=round,line width=\figuresLineThickness] (15.5797,3.0141) -- (15.5797,4.8078);
					\path[draw=black,line cap=round,line width=\figuresThickLineThickness] (17.3652,3.0222) -- (15.5903,3.0222);
					\path[draw=black,line cap=round,line width=\figuresLineThickness] (15.5903,3.0222) -- (13.7941,3.0222);
					\path[fill=black] (15.5797,3.0222) circle (\figuresSmallPointSize) node[below left] {2};
				\end{tikzpicture}
			\end{center}
		\end{minipage}
		\hfill
		\begin{minipage}{0.12\textwidth}
			\begin{center}
				\begin{tikzpicture}[scale=0.4, yscale=-1.0]
					\path[rotate=90.0,fill=lightgray] (1.2202,-19.6086)arc(180.000:-180.000:1.794) -- (3.0141,-19.6086) -- cycle;
					\path[draw=black,line cap=round,line width=\figuresThickLineThickness] (19.6086,1.2285) -- (19.6086,3.0059);
					\path[draw=black,line cap=round,line width=\figuresThickLineThickness] (19.6086,3.0059) -- (19.6086,4.7996);
					\path[draw=black,line cap=round,line width=\figuresThickLineThickness] (21.3942,3.0141) -- (19.6193,3.0141);
					\path[draw=black,line cap=round,line width=\figuresThickLineThickness] (19.6193,3.0141) -- (17.8230,3.0141);
					\path[fill=black] (19.6086,3.0141) circle (\figuresSmallPointSize) node[below left] {4};
				\end{tikzpicture}
			\end{center}
		\end{minipage}
		\caption{\label{Figure:VertexTypesNeighborhoods}All types of crossing neighbourhoods. $\gamma$-areas are shaded and $\gamma$-edges are drawn by thick lines}
	\end{center}
\end{figure}
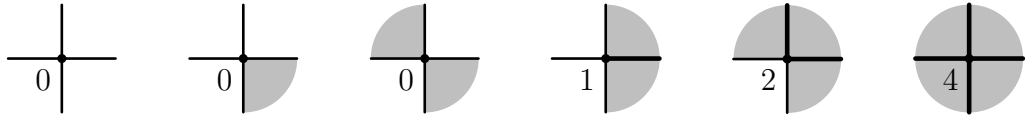

%% file: RightLeftExample.tex
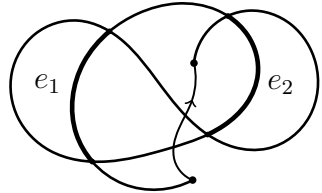
\begin{figure}[h]
	\begin{center}
		\begin{tikzpicture}[scale=0.25, yscale=-1.0]
			\path[draw=black,line cap=round,line width=\figuresLineThickness] (10.4835,11.4019) .. controls (9.3708,12.0013) and (7.9918,12.0814) .. (6.8172,11.6149) .. controls (6.1730,11.3591) and (5.5970,10.9415) .. (5.1423,10.4188);
			\path[draw=black,line cap=round,line width=\figuresThickLineThickness] (5.1423,10.4188) .. controls (4.7679,9.9882) and (4.4758,9.4863) .. (4.2959,8.9446) .. controls (3.8822,7.6987) and (4.0642,6.2890) .. (4.6925,5.1363) .. controls (5.0374,4.5034) and (5.5093,3.9474) .. (6.0628,3.4851);
			\path[draw=black,line cap=round,line width=\figuresLineThickness] (6.0628,3.4851) .. controls (6.5174,3.1053) and (7.0271,2.7888) .. (7.5665,2.5449) .. controls (8.2252,2.2470) and (8.9336,2.0512) .. (9.6552,2.0066) .. controls (10.3768,1.9619) and (11.1117,2.0711) .. (11.7750,2.3587) .. controls (11.9775,2.4465) and (12.1731,2.5512) .. (12.3584,2.6715);
			\path[draw=black,line cap=round,line width=\figuresThickLineThickness] (12.3584,2.6715) .. controls (12.7801,2.9453) and (13.1488,3.2998) .. (13.4253,3.7190) .. controls (13.8232,4.3226) and (14.0235,5.0610) .. (13.9485,5.7801) .. controls (13.8903,6.3391) and (13.6695,6.8755) .. (13.3476,7.3362) .. controls (13.0256,7.7969) and (12.6048,8.1838) .. (12.1391,8.4984) .. controls (11.8856,8.6696) and (11.6202,8.8193) .. (11.3463,8.9527);
			\path[draw=black,line cap=round,line width=\figuresLineThickness] (11.3463,8.9527) .. controls (10.6136,9.3095) and (9.8204,9.5499) .. (9.0338,9.7755) .. controls (8.0212,10.0659) and (6.9955,10.3408) .. (5.9443,10.4112) .. controls (4.8932,10.4817) and (3.8047,10.3362) .. (2.8816,9.8285) .. controls (2.1632,9.4334) and (1.5627,8.8215) .. (1.1979,8.0873) .. controls (0.8330,7.3530) and (0.7082,6.4987) .. (0.8662,5.6942) .. controls (1.0242,4.8896) and (1.4673,4.1410) .. (2.1082,3.6296) .. controls (2.7491,3.1183) and (3.5846,2.8519) .. (4.4020,2.9161) .. controls (5.3113,2.9876) and (6.1539,3.4517) .. (6.8374,4.0557) .. controls (7.5209,4.6597) and (8.0664,5.4009) .. (8.6096,6.1337) .. controls (9.3323,7.1087) and (10.0750,8.0939) .. (11.0488,8.8183) .. controls (11.1233,8.8738) and (11.1992,8.9276) .. (11.2763,8.9796);
			\path[draw=black,line cap=round,line width=\figuresLineThickness] (11.2763,8.9796) .. controls (11.7028,9.2674) and (12.1676,9.4995) .. (12.6612,9.6425) .. controls (13.2442,9.8113) and (13.8686,9.8529) .. (14.4612,9.7225) .. controls (15.4092,9.5138) and (16.2391,8.8577) .. (16.7023,8.0048) .. controls (17.1655,7.1518) and (17.2671,6.1189) .. (17.0247,5.1790) .. controls (16.8208,4.3888) and (16.3752,3.6566) .. (15.7417,3.1422) .. controls (15.1081,2.6278) and (14.2861,2.3403) .. (13.4712,2.3858) .. controls (12.7395,2.4266) and (12.0259,2.7366) .. (11.4966,3.2434) .. controls (10.9673,3.7502) and (10.6267,4.4498) .. (10.5542,5.1790);
			\path[fill=black] (10.4835,11.4019) circle (\figuresPointSize);
			\path[fill=black] (10.5542,5.1790) circle (\figuresPointSize);
			\path[fill=black] (5.1423,10.4188) circle (\figuresSmallPointSize);
			\path[fill=black] (6.0628,3.4851) circle (\figuresSmallPointSize);
			\path[fill=black] (12.3584,2.6715) circle (\figuresSmallPointSize);
			\path[fill=black] (11.2763,8.9796) circle (\figuresSmallPointSize);
			\path[draw=black, ->,line cap=round,line width=\figuresThinLineThickness] (10.4835,11.4019) .. controls (10.0321,11.2304) and (9.6692,10.8412) .. (9.5295,10.3790) .. controls (9.4456,10.1012) and (9.4395,9.8036) .. (9.4848,9.5170) .. controls (9.5301,9.2304) and (9.6255,8.9536) .. (9.7409,8.6874) .. controls (9.9716,8.1549) and (10.2840,7.6571) .. (10.4658,7.1060);
			\path[draw=black,line cap=round,line width=\figuresThinLineThickness] (10.4658,7.1060) .. controls (10.6698,6.4877) and (10.7008,5.8134) .. (10.5542,5.1790);
			\draw (4.0642,6.2890) node[left] {$e_1$};
			\draw (13.8903,6.3391) node[right] {$e_2$};
		\end{tikzpicture}
		\caption{\label{Figure:RightLeftExample}The left border edge $e_1$ and the right border edge $e_2$}
	\end{center}
\end{figure}

%% file: TwoPointArc.tex
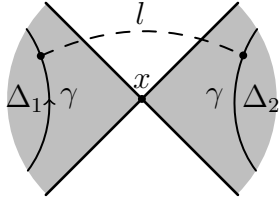
\begin{figure}[h]
	\begin{center}
		\begin{tikzpicture}[scale=0.35, yscale=-1.0]
			\path[rotate=-45.0,fill=lightgray] (-2.9377,8.6459)arc(180.000:270.000:5.120) -- (2.1824,8.6459) -- cycle;
			\path[cm={{-0.70711,-0.70711,-0.70711,0.70711,(0.0,0.0)}},fill=lightgray] (-13.7660,-2.1824)arc(180.000:270.000:5.120) -- (-8.6459,-2.1824) -- cycle;
			\path[draw=black,line cap=round,line width=\figuresLineThickness] (4.0090,0.9226) -- (11.3045,8.2182);
			\path[draw=black,line cap=round,line width=\figuresLineThickness] (11.3045,0.9226) -- (4.0090,8.2182);
			\path[fill=black] (7.6567,4.5704) circle (\figuresSmallPointSize) node[above] {$x$};
			\path[draw=black,->,line cap=round,line width=\figuresThinLineThickness] (3.3231,7.2971) .. controls (3.9411,6.4767) and (4.1709,5.4947) .. (4.1369,4.5521);
			\path[draw=black,line cap=round,line width=\figuresThinLineThickness] (4.1369,4.5521) .. controls (4.1006,3.5468) and (3.7643,2.5864) .. (3.2790,1.9150);
			\path[draw=black,line cap=round,line width=\figuresThinLineThickness] (11.9904,7.2971) .. controls (11.3723,6.4767) and (11.1426,5.4947) .. (11.1766,4.5521) .. controls (11.2128,3.5468) and (11.5491,2.5864) .. (12.0344,1.9150);
			\path[draw=black,dash pattern=on 5.0 off 5.0,line cap=round,line width=\figuresThinLineThickness] (3.8157,2.9362) .. controls (6.5030,1.7134) and (8.8847,1.6555) .. (11.5053,2.9161);
			\draw (7.6, 1.3) node {$l$};
			\draw (4.1369,4.5521) node[right] {$\gamma$};
			\draw (11.1766,4.5521) node[left] {$\gamma$};
			\draw (3.2,4.5704) node {$\Delta_1$};
			\draw (12.2,4.5704) node {$\Delta_2$};
			\path[fill=black] (3.8157,2.9362) circle (\figuresSmallPointSize);
			\path[fill=black] (11.5053,2.9161) circle (\figuresSmallPointSize);
		\end{tikzpicture}
		\caption{\label{Figure:TwoPointArc} The simple arc $l$ going from $\Delta_1$ to $\Delta_2$ and crossing $F$ twice nearby $x$}
	\end{center}
\end{figure}

%% file: RightChainExample.tex
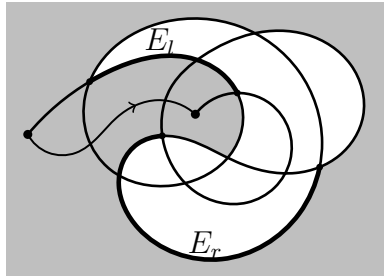
\begin{figure}[h]
	\begin{center}
		\begin{tikzpicture}[scale=0.3, yscale=-1.0]
			\path[fill=lightgray] (1.9988,0.6490) -- (18.7319,0.6490) -- (18.7319,12.7818) -- (1.9988,12.7818) -- cycle;
			\path[fill=white, line width=0.0] (5.6629,4.1391) .. controls (5.6629,4.1391) and (5.8707,3.5535) .. (6.4335,2.9351) .. controls (6.7550,2.5818) and (6.9704,2.3302) .. (7.9380,1.8390) .. controls (8.9056,1.3478) and (10.9630,1.3394) .. (11.6007,1.5015) .. controls (12.3315,1.6872) and (13.0258,2.0140) .. (13.0258,2.0140) .. controls (13.0258,2.0140) and (14.5009,1.7390) .. (15.4135,2.0765) .. controls (16.3260,2.4140) and (17.4136,3.2516) .. (17.7386,4.5142) .. controls (18.0636,5.7767) and (17.5011,6.4143) .. (17.3136,6.7768) .. controls (17.1261,7.1393) and (16.7114,7.4752) .. (15.8510,7.9894) .. controls (15.6260,8.7644) and (14.9459,10.1261) .. (14.3948,10.6770) .. controls (13.7259,11.3458) and (12.2107,12.1151) .. (11.0320,12.0418) .. controls (9.8534,11.9684) and (8.3321,11.3781) .. (7.6130,10.4270) .. controls (6.8504,9.4186) and (6.8113,8.3002) .. (7.2563,7.5152) .. controls (7.6536,6.8143) and (8.1197,6.5341) .. (9.0935,6.5695) .. controls (10.3348,6.6145) and (11.8132,7.5393) .. (11.8132,7.5393) .. controls (11.8132,7.5393) and (12.3477,6.7680) .. (12.4477,6.1804) .. controls (12.5477,5.5929) and (12.2309,4.4172) .. (11.7507,3.9891) .. controls (11.3062,3.5929) and (10.8625,3.2340) .. (10.0163,3.0750) .. controls (9.1700,2.9161) and (8.8594,3.0016) .. (7.7718,3.2391) .. controls (6.6877,3.4758) and (5.6629,4.1391) .. (5.6629,4.1391) -- cycle;
			\path[draw=black,line cap=round,line width=\figuresLineThickness] (2.9341,6.4976) .. controls (3.7147,5.5594) and (4.6587,4.7579) .. (5.7095,4.1372);
			\path[draw=black,line cap=round,line width=\figuresThickLineThickness] (5.7095,4.1372) .. controls (5.9721,3.9821) and (6.2415,3.8382) .. (6.5165,3.7064) .. controls (7.5772,3.1980) and (8.7698,2.8640) .. (9.9286,3.0655) .. controls (10.5080,3.1663) and (11.0683,3.4033) .. (11.5191,3.7810) .. controls (11.7994,4.0159) and (12.0355,4.3065) .. (12.2026,4.6312);
			\path[draw=black,line cap=round,line width=\figuresLineThickness] (12.2026,4.6312) .. controls (12.3042,4.8286) and (12.3804,5.0387) .. (12.4255,5.2564) .. controls (12.5141,5.6841) and (12.4824,6.1343) .. (12.3483,6.5499) .. controls (12.2142,6.9656) and (11.9792,7.3467) .. (11.6782,7.6632) .. controls (11.0763,8.2963) and (10.2237,8.6579) .. (9.3562,8.7613) .. controls (8.3859,8.8771) and (7.3522,8.6679) .. (6.5898,8.0566) .. controls (6.1172,7.6778) and (5.7634,7.1559) .. (5.5731,6.5809) .. controls (5.3829,6.0059) and (5.3551,5.3799) .. (5.4799,4.7872) .. controls (5.6538,3.9618) and (6.1212,3.2097) .. (6.7496,2.6470) .. controls (7.3780,2.0843) and (8.1615,1.7069) .. (8.9823,1.5124) .. controls (10.0748,1.2536) and (11.2427,1.3164) .. (12.2979,1.6999) .. controls (13.3531,2.0834) and (14.2908,2.7874) .. (14.9459,3.6992) .. controls (15.6010,4.6109) and (15.9693,5.7273) .. (15.9760,6.8500) .. controls (15.9782,7.2280) and (15.9393,7.6061) .. (15.8615,7.9760);
			\path[draw=black,line cap=round,line width=\figuresThickLineThickness] (15.8615,7.9760) .. controls (15.7081,8.7047) and (15.4037,9.4018) .. (14.9656,10.0039) .. controls (14.1605,11.1104) and (12.8915,11.8831) .. (11.5307,12.0270) .. controls (10.1699,12.1708) and (8.7418,11.6629) .. (7.8229,10.6491) .. controls (7.3484,10.1257) and (7.0063,9.4601) .. (6.9683,8.7548) .. controls (6.9493,8.4021) and (7.0072,8.0438) .. (7.1498,7.7207) .. controls (7.2924,7.3976) and (7.5208,7.1107) .. (7.8122,6.9111) .. controls (8.1189,6.7011) and (8.4879,6.5920) .. (8.8592,6.5679);
			\path[draw=black,line cap=round,line width=\figuresLineThickness] (8.8592,6.5679) .. controls (8.8771,6.5667) and (8.8950,6.5658) .. (8.9129,6.5650) .. controls (9.3022,6.5483) and (9.6915,6.6209) .. (10.0623,6.7406) .. controls (10.8039,6.9799) and (11.4714,7.4022) .. (12.1848,7.7157) .. controls (13.0108,8.0786) and (13.9185,8.2960) .. (14.8163,8.2062) .. controls (15.5098,8.1369) and (16.1892,7.8796) .. (16.7299,7.4399) .. controls (17.2706,7.0001) and (17.6659,6.3746) .. (17.7921,5.6891) .. controls (17.9380,4.8969) and (17.7158,4.0547) .. (17.2417,3.4034) .. controls (16.7675,2.7521) and (16.0552,2.2896) .. (15.2812,2.0661) .. controls (14.5072,1.8427) and (13.6754,1.8510) .. (12.8907,2.0335) .. controls (12.1060,2.2160) and (11.3668,2.5683) .. (10.6962,3.0146) .. controls (10.1349,3.3881) and (9.6080,3.8404) .. (9.2718,4.4247) .. controls (8.8233,5.2041) and (8.7555,6.1704) .. (8.9920,7.0379) .. controls (9.2065,7.8244) and (9.6711,8.5511) .. (10.3341,9.0255) .. controls (10.9970,9.4998) and (11.8607,9.7052) .. (12.6545,9.5198) .. controls (13.1100,9.4134) and (13.5360,9.1819) .. (13.8710,8.8554) .. controls (14.2059,8.5290) and (14.4485,8.1081) .. (14.5601,7.6539) .. controls (14.6716,7.1996) and (14.6514,6.7134) .. (14.4995,6.2710) .. controls (14.3477,5.8286) and (14.0643,5.4318) .. (13.6931,5.1471) .. controls (13.1984,4.7676) and (12.5494,4.5956) .. (11.9316,4.6801) .. controls (11.3138,4.7647) and (10.7349,5.1049) .. (10.3604,5.6034);
			\path[draw=black,->,line cap=round,line width=\figuresThinLineThickness] (2.9341,6.4976) .. controls (3.1272,6.8649) and (3.4572,7.1582) .. (3.8446,7.3069) .. controls (4.2321,7.4555) and (4.6736,7.4582) .. (5.0628,7.3143) .. controls (5.4402,7.1748) and (5.7570,6.9083) .. (6.0383,6.6207) .. controls (6.3197,6.3331) and (6.5746,6.0189) .. (6.8754,5.7517) .. controls (7.1350,5.5211) and (7.4320,5.3271) .. (7.7512,5.1908);
			\path[draw=black,line cap=round,line width=\figuresThinLineThickness] (7.7512,5.1908) .. controls (8.0456,5.0652) and (8.3588,4.9888) .. (8.6789,4.9783) .. controls (9.2918,4.9582) and (9.9094,5.1878) .. (10.3604,5.6034);
			\path[fill=black] (2.9341,6.4976) circle (\figuresPointSize);
			\path[fill=black] (10.3604,5.6034) circle (\figuresPointSize);
			\path[fill=black] (5.6735,4.1586) circle (\figuresSmallPointSize);
			\path[fill=black] (12.2168,4.6592) circle (\figuresSmallPointSize);
			\path[fill=black] (15.8675,7.9471) circle (\figuresSmallPointSize);
			\path[fill=black] (8.8937,6.5573) circle (\figuresSmallPointSize);
			\draw (10.8, 11.4) node {$E_r$};
			\draw (8.8, 2.4) node {$E_l$};
		\end{tikzpicture}
		\caption{\label{Figure:RightChainExample}$E_l$ is a left border chain and $E_r$ is a right border chain}
	\end{center}
\end{figure}

%% file: ChainExeptions.tex
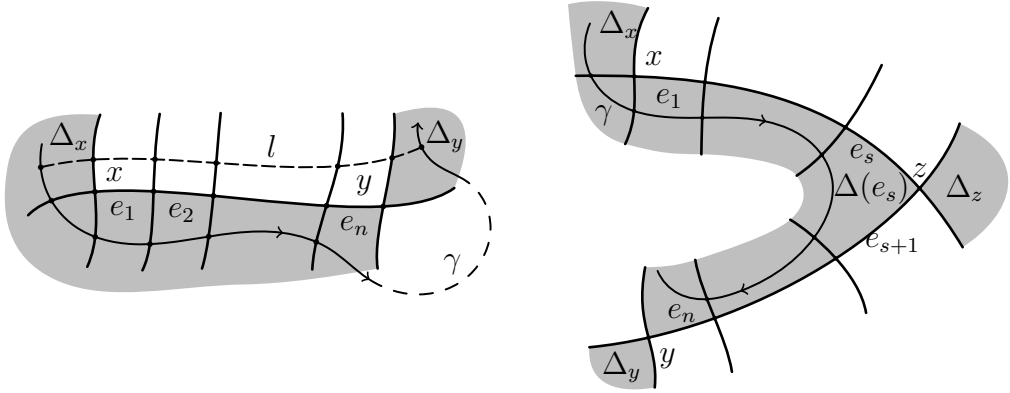
\begin{figure}[h]
	\begin{center}
		\begin{minipage}{0.45\textwidth}
			\begin{center}
				\begin{tikzpicture}[scale=0.25, yscale=-1.0]
					\path[fill=lightgray] (19.6992,3.8046) .. controls (19.8404,5.2985) and (19.4372,7.0852) .. (19.1199,8.7296) .. controls (20.7024,8.6485) and (22.0271,8.3708) .. (22.9010,7.7494) .. controls (24.0631,4.8833) and (23.8199,2.5749) .. (19.6992,3.8046) -- cycle;
					\path[fill=lightgray] (18.8253,12.0404) .. controls (16.4788,12.4524) and (14.0444,12.8405) .. (11.5987,12.8833) .. controls (8.2055,12.9427) and (4.5396,13.9002) .. (1.7727,12.6704) .. controls (-0.9942,11.4407) and (-1.2070,9.0049) .. (-0.8996,6.8766) .. controls (-0.5922,4.7482) and (0.5193,3.5894) .. (4.0686,3.8580) .. controls (3.5705,5.1983) and (3.6910,6.5743) .. (3.7885,7.9474) .. controls (8.0610,7.5932) and (14.6316,8.9602) .. (19.1219,8.7341) .. controls (18.8829,9.9792) and (18.6928,11.1438) .. (18.8253,12.0404);
					\path[draw=black,line cap=round,line width=\figuresLineThickness] (0.2388,9.1953) .. controls (3.6915,5.6480) and (18.6442,10.7764) .. (22.9010,7.7494);
					\path[draw=black,line cap=round,line width=\figuresLineThickness] (3.3817,11.9397) .. controls (4.5405,9.3384) and (3.0743,6.5242) .. (4.0675,3.8519);
					\path[draw=black,line cap=round,line width=\figuresLineThickness] (19.6992,3.8046) .. controls (19.9472,6.4276) and (18.5168,9.9532) .. (18.8242,12.0343);
					\path[draw=black,line cap=round,line width=\figuresLineThickness] (7.2364,3.8046) .. controls (6.6925,6.8316) and (7.1655,10.4735) .. (6.3142,12.1289);
					\path[draw=black,line cap=round,line width=\figuresLineThickness] (10.4763,3.7573) .. controls (10.2398,5.9566) and (10.0270,10.3553) .. (9.4831,12.0343);
					\path[draw=black,line cap=round,line width=\figuresLineThickness] (16.9560,3.7810) .. controls (17.1925,6.8789) and (15.4898,9.3857) .. (15.3479,12.0816);
					\path[fill=black] (3.7885,7.9474) circle (\figuresSmallPointSize) node[above right] {$x$};
					\path[fill=black] (6.9100,7.9258) circle (\figuresSmallPointSize);
					\draw (5.3, 7.9) node[below] {$e_1$};
					\path[fill=black] (10.0880,8.1639) circle (\figuresSmallPointSize);
					\draw (8.4, 8.0) node[below] {$e_2$};
					\path[fill=black] (16.1039,8.6932) circle (\figuresSmallPointSize);
					\path[fill=black] (19.1213,8.7280) circle (\figuresSmallPointSize) node[above left] {$y$};
					\draw (17.6, 8.7) node[below] {$e_n$};
					\path[draw=black, ->,line cap=round,line width=\figuresThinLineThickness] (0.9696,5.4127) .. controls (0.8103,6.2216) and (0.8888,7.0759) .. (1.1928,7.8423) .. controls (1.4969,8.6086) and (2.0254,9.2843) .. (2.6959,9.7640) .. controls (3.5882,10.4023) and (4.6982,10.6798) .. (5.7939,10.7336) .. controls (7.6106,10.8230) and (9.4010,10.3372) .. (11.2094,10.1424) .. controls (12.0762,10.0491) and (12.9546,10.0228) .. (13.8182,10.1419);
					\path[draw=black,->,line cap=round,line width=\figuresThinLineThickness] (13.8182,10.1419) .. controls (14.6819,10.2609) and (15.5325,10.5292) .. (16.2702,10.9938) .. controls (17.0304,11.4725) and (17.6526,12.1455) .. (18.4083,12.6312);
					\path[draw=black,dash pattern=on 5.0 off 5.0,line cap=round,line width=\figuresThinLineThickness] (18.4083,12.6312) .. controls (19.2058,13.1438) and (20.1520,13.4359) .. (21.0997,13.4100) .. controls (22.0474,13.3841) and (22.9914,13.0314) .. (23.6905,12.3910) .. controls (24.3895,11.7506) and (24.8278,10.8205) .. (24.8270,9.8725) .. controls (24.8266,9.3985) and (24.7190,8.9241) .. (24.5089,8.4992) .. controls (24.2987,8.0743) and (23.9859,7.6998) .. (23.6012,7.4228);
					\draw (22.8,11.8) node {$\gamma$};
					\path[draw=black,->,line cap=round,line width=\figuresThinLineThickness] (23.6012,7.4228) .. controls (23.0151,7.0007) and (22.2779,6.8092) .. (21.7330,6.3350) .. controls (21.4423,6.0821) and (21.2173,5.7544) .. (21.0856,5.3923) .. controls (20.9539,5.0302) and (20.9159,4.6345) .. (20.9762,4.2539);
					\path[fill=black] (1.4744,8.4203) circle (\figuresSmallPointSize);
					\path[fill=black] (3.8038,10.3463) circle (\figuresSmallPointSize);
					\path[fill=black] (6.7152,10.7337) circle (\figuresSmallPointSize);
					\path[fill=black] (9.8358,10.3327) circle (\figuresSmallPointSize);
					\path[fill=black] (15.5603,10.6194) circle (\figuresSmallPointSize);
					\path[draw=black,dash pattern=on 5.0 off 2.0,line cap=round,line width=\figuresThinLineThickness] (0.9077,6.6460) .. controls (7.0601,5.1179) and (15.7845,8.1162) .. (21.1591,5.5702);
					\path[fill=black] (0.9077,6.6460) circle (\figuresSmallPointSize);
					\path[fill=black] (3.7001,6.2379) circle (\figuresSmallPointSize);
					\path[fill=black] (6.9786,6.2209) circle (\figuresSmallPointSize);
					\path[fill=black] (10.2422,6.4156) circle (\figuresSmallPointSize);
					\path[fill=black] (16.7039,6.5905) circle (\figuresSmallPointSize);
					\path[fill=black] (19.5917,6.1409) circle (\figuresSmallPointSize);
					\path[fill=black] (21.1591,5.5702) circle (\figuresSmallPointSize);
					\draw (2.4, 5.0) node {$\Delta_x$};
					\draw (22.4, 5.0) node {$\Delta_y$};
					\draw (13.2, 5.5) node {$l$};
				\end{tikzpicture}
			\end{center}
		\end{minipage}
		\hfill
		\begin{minipage}{0.45\textwidth}
			\begin{center}
				\begin{tikzpicture}[scale=0.2, yscale=-1.0]
					\path[fill=lightgray](24.3924,13.0882) .. controls (25.6340,11.6695) and (26.5558,10.2135) .. (27.0493,8.7758) .. controls (32.1067,11.1776) and (30.4868,14.9357) .. (27.2500,17.0031) .. controls (26.3596,15.5277) and (25.4108,14.2319) .. (24.3924,13.0882) -- cycle;
					\path[fill=lightgray] (9.4837,17.8499) .. controls (10.9620,17.3751) and (14.6388,16.2081) .. (15.7979,15.1776) .. controls (16.9569,14.1470) and (16.8499,13.1338) .. (16.2235,12.2452) .. controls (15.5972,11.3565) and (12.0010,10.8502) .. (9.8857,10.8262) .. controls (7.7705,10.8023) and (7.2399,10.7773) .. (4.8249,10.1404) .. controls (2.4100,9.5036) and (1.7429,7.2201) .. (1.5391,5.6043) .. controls (1.3352,3.9885) and (1.1721,2.6492) .. (1.0239,0.8553) .. controls (2.2402,0.5097) and (4.6485,0.5738) .. (6.3384,1.1067) .. controls (5.4208,2.7928) and (5.3822,4.2397) .. (5.4162,5.6384) .. controls (13.4357,5.9347) and (19.6144,7.7290) .. (24.3924,13.0882) .. controls (20.5160,17.4661) and (13.4987,21.4087) .. (6.3729,23.0173) .. controls (6.0536,21.7913) and (5.7537,20.4648) .. (6.1493,18.3465) .. controls (7.1546,18.3092) and (8.0054,18.3246) .. (9.4837,17.8499) -- cycle;
					\path[fill=lightgray] (2.4414,23.6661) .. controls (3.7529,23.5382) and (5.0505,23.3379) .. (6.3740,23.0371) .. controls (6.6413,24.0713) and (6.9169,25.0335) .. (6.7871,26.3417) .. controls (3.7671,26.8561) and (2.2906,25.7603) .. (2.4414,23.6661) -- cycle;
					\path[draw=black,line cap=round,line width=\figuresLineThickness] (1.5384,5.6063) .. controls (13.3107,5.4391) and (21.5045,7.4792) .. (27.2569,17.0107);
					\path[draw=black,line cap=round,line width=\figuresLineThickness] (2.4414,23.6661) .. controls (13.0766,22.6293) and (24.7486,15.5057) .. (27.0562,8.7835);
					\path[draw=black,line cap=round,line width=\figuresLineThickness] (6.3378,1.1087) .. controls (4.4459,4.5850) and (6.2905,7.0445) .. (4.8243,10.1424);
					\path[draw=black,line cap=round,line width=\figuresLineThickness] (11.0202,1.7708) .. controls (9.8851,5.2235) and (9.9324,8.6999) .. (9.8851,10.8282);
					\path[draw=black,line cap=round,line width=\figuresLineThickness] (16.2229,12.2471) .. controls (19.2084,10.2494) and (21.0708,6.5951) .. (21.8749,4.8924);
					\path[draw=black,line cap=round,line width=\figuresLineThickness] (15.7972,15.1796) .. controls (17.3344,16.1018) and (20.4087,18.7741) .. (20.8344,21.3755);
					\path[draw=black,line cap=round,line width=\figuresLineThickness] (6.1486,18.3485) .. controls (5.4155,22.2741) and (7.0709,23.4802) .. (6.7871,26.3417);
					\path[fill=black] (5.4180,5.6427) circle (\figuresSmallPointSize) node[above right] {$x$};
					\draw (4.4, 2.2) node {$\Delta_x$};
					\draw (3.4, 8.2) node {$\gamma$};
					\path[fill=black] (24.3924,13.0882) circle (\figuresSmallPointSize) node[above] {$z$};
					\draw (27.3924,13.0882) node {$\Delta_z$};
					\draw (21.2,13.0882) node {$\Delta(e_s)$};
					\path[fill=black] (10.1212,6.0308) circle (\figuresSmallPointSize);
					\draw (7.75, 5.8) node[below] {$e_1$};
					\path[fill=black] (19.4616,9.0310) circle (\figuresSmallPointSize);
					\path[fill=black] (18.8297,17.7552) circle (\figuresSmallPointSize);
					\draw (20.6, 10.8) node {$e_s$};
					\path[fill=black] (6.3707,23.0245) circle (\figuresSmallPointSize) node[below right] {$y$};
					\draw (4.6, 25.0) node {$\Delta_y$};
					\draw (22.6, 16.8) node {$e_{s+1}$};
					\path[fill=black] (2.5381,5.5975) circle (\figuresSmallPointSize);
					\path[fill=black] (5.4099,7.9508) circle (\figuresSmallPointSize);
					\path[fill=black] (9.9440,8.3986) circle (\figuresSmallPointSize);
					\path[fill=black] (17.8917,10.8669) circle (\figuresSmallPointSize);
					\path[fill=black] (17.3923,16.3551) circle (\figuresSmallPointSize);
					\path[fill=black] (10.2667,20.4650) circle (\figuresSmallPointSize);
					\path[fill=black] (10.7705,21.7348) circle (\figuresSmallPointSize);
					\draw (8.6, 21.4) node {$e_{n}$};
					\path[draw=black, ->,line cap=round,line width=\figuresThinLineThickness] (2.4831,2.1256) .. controls (2.1603,2.8506) and (2.0529,3.6700) .. (2.1780,4.4538) .. controls (2.3031,5.2376) and (2.6601,5.9829) .. (3.1926,6.5715) .. controls (3.7629,7.2020) and (4.5209,7.6446) .. (5.3246,7.9220) .. controls (6.1283,8.1994) and (6.9789,8.3184) .. (7.8277,8.3688) .. controls (9.9628,8.4956) and (12.1362,8.2003) .. (14.2259,8.6559);
					\path[draw=black, ->,line cap=round,line width=\figuresThinLineThickness] (14.2259,8.6559) .. controls (15.2708,8.8837) and (16.2893,9.3072) .. (17.1025,10.0017) .. controls (17.9158,10.6962) and (18.5093,11.6792) .. (18.6114,12.7438) .. controls (18.7069,13.7396) and (18.3743,14.7364) .. (17.8783,15.6052) .. controls (16.7098,17.6519) and (14.6621,19.0859) .. (12.4628,19.9329);
					\path[draw=black,line cap=round,line width=\figuresThinLineThickness] (12.4628,19.9329) .. controls (11.8061,20.1858) and (11.1254,20.3936) .. (10.4243,20.4541) .. controls (9.7232,20.5146) and (8.9970,20.4211) .. (8.3716,20.0984) .. controls (7.7492,19.7773) and (7.2458,19.2314) .. (6.9763,18.5849);
					\path[draw=black,line cap=round,line width=\figuresLineThickness] (9.4831,17.8518) .. controls (9.8797,20.2107) and (11.5641,23.0072) .. (11.9189,25.3011);
				\end{tikzpicture}
			\end{center}
		\end{minipage}
		\caption{\label{Figure:ChainExeptions} The neighborhood of the left border chain $\{e_1, \ldots, e_n\}$. This chain goes through crossings of the type 1 only (on the left) or through one two-sided exceptional crossing of the type 0 (on the right)}
	\end{center}
\end{figure}

%% file: LeftRightVertex.tex
\begin{figure}[h]
	\begin{center}
		\begin{minipage}{0.45\textwidth}
			\begin{center}
				\begin{tikzpicture}[scale=0.3, yscale=-1.0]
					\path[rotate=-45.0,fill=lightgray] (5.2631,7.6748)arc(0.000:67.500:4.762)arc(67.500:135.000:4.762)arc(135.000:202.500:4.762)arc(202.500:270.000:4.762) -- (0.5006,7.6748) -- cycle;
					\path[draw=black,line cap=round,line width=\figuresLineThickness] (2.1922,1.4842) -- (5.7988,5.0908);
					\path[draw=black,line cap=round,line width=\figuresThickLineThickness] (5.7988,5.0908) -- (9.3697,8.6617);
					\draw (3.8, 5.9) node {$e_1$};
					\draw (7.3, 5.6) node {$e_2$};
					\path[draw=black,line cap=round,line width=\figuresLineThickness] (9.3697,1.4842) -- (5.7638,5.0901);
					\path[fill=black] (5.7810,5.0729) circle (\figuresSmallPointSize) node[above] {$x$};
					\path[draw=black,line cap=round,line width=\figuresThickLineThickness] (5.7638,5.0901) -- (3.1360,7.7179);
					\path[fill=black] (3.1360,7.7179) circle (\figuresPointSize);
					\path[draw=black,->,line cap=round,line width=\figuresThinLineThickness] (3.1360,7.7179) .. controls (5.4628,8.3894) and (7.5880,8.2269) .. (9.7756,4.5017);
					\draw (5.9, 8.8) node {$\gamma$};
					\path[fill=black] (7.7270,7.0190) circle (\figuresSmallPointSize);
				\end{tikzpicture}
			\end{center}
		\end{minipage}
		\hfill
		\begin{minipage}{0.45\textwidth}
			\begin{center}
				\begin{tikzpicture}[scale=0.3, yscale=-1.0]
					\path[cm={{-0.70711,-0.70711,-0.70711,0.70711,(0.0,0.0)}},fill=lightgray] (-11.4126,-9.0262)arc(-0.000:67.500:4.762)arc(67.500:135.000:4.762)arc(135.000:202.500:4.762)arc(202.500:270.000:4.762) -- (-16.1751,-9.0262) -- cycle;
					\path[draw=black,line cap=round,line width=\figuresLineThickness] (21.4088,1.4663) -- (17.8021,5.0729);
					\path[draw=black,line cap=round,line width=\figuresThickLineThickness] (17.8021,5.0729) -- (14.2312,8.6439);
					\path[draw=black,line cap=round,line width=\figuresLineThickness] (14.2312,1.4663) -- (17.8371,5.0723);
					\path[xscale=-1.000,yscale=1.000,fill=black] (-17.8200,5.0551) circle (\figuresSmallPointSize) node[above] {$x$};
					\draw (19.8,5.9) node {$e_1$};
					\draw (16.4,5.5) node {$e_2$};
					\path[draw=black,line cap=round,line width=\figuresThickLineThickness] (17.8371,5.0723) -- (20.4649,7.7001);
					\path[xscale=-1.000,yscale=1.000,fill=black] (-20.4649,7.7001) circle (\figuresPointSize);
					\path[draw=black,->,line cap=round,line width=\figuresThinLineThickness] (20.4649,7.7001) .. controls (18.1381,8.3716) and (16.0130,8.2090) .. (13.8253,4.4838);
					\draw (18.1, 8.8) node {$\gamma$};
					\path[xscale=-1.000,yscale=1.000,fill=black] (-15.8739,7.0012) circle (\figuresSmallPointSize);
				\end{tikzpicture}
			\end{center}
		\end{minipage}
		\caption{\label{Figure:LeftRightVertex}Left-sided crossing of the type two (on the left) and right-sided one (on the right)}
	\end{center}
\end{figure}
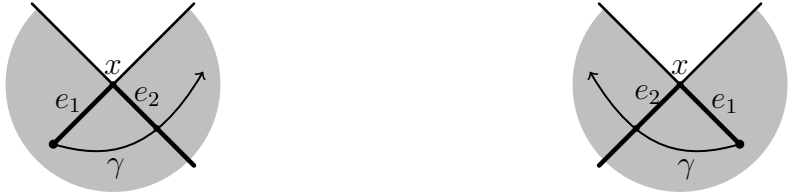

%% file: EndedVertexZero.tex
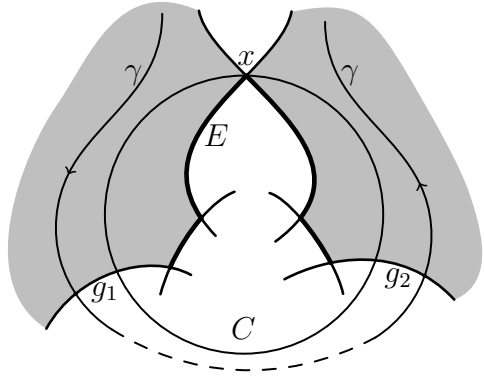
\begin{figure}[h]
	\begin{center}
		\begin{tikzpicture}[scale=0.2, yscale=-1.0]
			\path[fill=lightgray] (33.1187,14.7239) .. controls (32.4832,12.8510) and (30.6772,9.6738) .. (29.1388,7.0318) .. controls (26.4975,2.4956) and (25.4265,1.8814) .. (21.3026,2.6604) .. controls (21.0471,4.0492) and (19.7248,5.4491) .. (18.2960,6.9341) .. controls (21.1347,9.8195) and (24.4781,13.0397) .. (21.8900,16.3859) .. controls (22.6287,17.2478) and (23.3645,18.3290) .. (23.9055,19.3706) .. controls (26.5345,18.8430) and (29.6774,19.1813) .. (32.1186,21.8241) .. controls (35.0985,20.6301) and (33.7541,16.5968) .. (33.1187,14.7239) -- cycle;
			\path[fill=lightgray] (5.7169,8.8515) .. controls (7.8729,5.7421) and (9.7270,0.1890) .. (15.1455,2.5974) .. controls (15.3899,3.9256) and (16.7754,5.3822) .. (18.2960,6.9341) .. controls (15.6905,9.6418) and (12.7523,12.6394) .. (15.2919,16.3889) .. controls (14.5235,17.2827) and (13.7843,18.4236) .. (13.2294,19.6756) .. controls (10.4430,19.1183) and (6.7506,20.5852) .. (5.0011,23.7485) .. controls (0.0625,20.6104) and (3.4127,12.1746) .. (5.7169,8.8515) -- cycle;
			\path[draw=black,line cap=round,line width=\figuresLineThickness] (15.1502,2.5963) .. controls (15.3941,3.9221) and (16.7748,5.3756) .. (18.2994,6.9249);
			\path[draw=black,line cap=round,line width=\figuresThickLineThickness] (18.2994,6.9249) .. controls (21.1403,9.8121) and (24.4807,13.0320) .. (21.8949,16.3775);
			\path[draw=black,line cap=round,line width=\figuresLineThickness] (21.8949,16.3775) .. controls (21.4926,16.8979) and (20.9470,17.4214) .. (20.2337,17.9472);
			\path[draw=black,line cap=round,line width=\figuresLineThickness] (21.3046,2.6494) .. controls (21.0485,4.0410) and (19.7215,5.4437) .. (18.2918,6.9340);
			\path[draw=black,line cap=round,line width=\figuresThickLineThickness] (18.2918,6.9340) .. controls (15.6910,9.6452) and (12.7505,12.6463) .. (15.3006,16.3977);
			\path[draw=black,line cap=round,line width=\figuresLineThickness] (15.3006,16.3977) .. controls (15.5598,16.7791) and (15.8757,17.1681) .. (16.2545,17.5654);
			\path[draw=black,line cap=round,line width=\figuresLineThickness] (17.5120,14.6398) .. controls (16.8306,14.9031) and (16.0480,15.5317) .. (15.2979,16.3958);
			\path[draw=black,line cap=round,line width=\figuresThickLineThickness] (15.2979,16.3958) .. controls (14.5243,17.2871) and (13.7853,18.4289) .. (13.2278,19.6791);
			\path[draw=black,line cap=round,line width=\figuresLineThickness] (13.2278,19.6791) .. controls (12.9721,20.2527) and (12.7545,20.8492) .. (12.5893,21.4547);
			\path[draw=black,line cap=round,line width=\figuresLineThickness] (14.6211,20.1686) .. controls (11.8194,18.6501) and (7.0658,20.0227) .. (5.0058,23.7475);
			\path[draw=black,line cap=round,line width=\figuresLineThickness] (19.8457,14.7247) .. controls (20.4203,14.8951) and (21.1603,15.5323) .. (21.8921,16.3811);
			\path[draw=black,line cap=round,line width=\figuresThickLineThickness] (21.8921,16.3811) .. controls (22.6376,17.2460) and (23.3747,18.3307) .. (23.9196,19.3653);
			\path[draw=black,line cap=round,line width=\figuresLineThickness] (23.9196,19.3653) .. controls (24.3120,20.1105) and (24.6048,20.8297) .. (24.7292,21.4222);
			\path[draw=black,line cap=round,line width=\figuresLineThickness] (20.8340,20.4592) .. controls (23.6442,18.9787) and (28.6085,18.0110) .. (32.1206,21.8130);
			\path[draw=black,line cap=round,line width=\figuresThinLineThickness] (18.1469,16.1600) circle (9.25);
			\draw (18.1, 23.7) node {$C$};
			\path[draw=black,line cap=round,line width=\figuresThinLineThickness] (9.3643,23.7999) .. controls (7.7131,22.7557) and (6.4677,21.0869) .. (5.9366,19.2068) .. controls (5.4055,17.3267) and (5.5938,15.2528) .. (6.4547,13.4991);
			\path[draw=black, <-,line cap=round,line width=\figuresThinLineThickness] (6.4547,13.4991) .. controls (7.1438,12.0954) and (8.2238,10.9286) .. (9.2793,9.7748) .. controls (10.3348,8.6210) and (11.3974,7.4333) .. (12.0399,6.0076) .. controls (12.4925,5.0030) and (12.7226,3.8988) .. (12.7087,2.7970);
			\path[fill=black] (18.3000,6.9255) circle (\figuresSmallPointSize) node[above] {$x$};
			\draw (10.8, 6.9) node {$\gamma$};
			\draw (25.2, 6.9) node {$\gamma$};
			\draw (16.3, 10.9) node {$E$};
			\draw (8.9, 21.3) node {$g_1$};
			\draw (28.4, 20.4) node {$g_2$};
			\path[fill=black] (15.2987,16.3949) circle (\figuresSmallPointSize);
			\path[fill=black] (21.8921,16.3812) circle (\figuresSmallPointSize);
			\path[fill=black] (23.9196,19.3653) circle (\figuresSmallPointSize);
			\path[fill=black] (13.2294,19.6756) circle (\figuresSmallPointSize);
			\path[fill=black] (6.9558,21.4256) circle (\figuresSmallPointSize);
			\path[fill=black] (30.1531,20.2268) circle (\figuresSmallPointSize);
			\path[draw=black,dash pattern=on 5.0 off 5.0,line cap=round,line width=\figuresThinLineThickness] (9.3643,23.7999) .. controls (14.2806,26.9102) and (20.5830,27.6792) .. (26.8798,24.3349);
			\path[draw=black,line cap=round,line width=\figuresThinLineThickness] (29.7894,14.0341) .. controls (29.1003,12.6304) and (28.0203,11.4637) .. (26.9648,10.3098) .. controls (25.9093,9.1560) and (24.8466,7.9683) .. (24.2042,6.5426) .. controls (23.7516,5.5380) and (23.5215,4.4338) .. (23.5354,3.3320);
			\path[draw=black, ->,line cap=round,line width=\figuresThinLineThickness] (26.8798,24.3349) .. controls (28.5310,23.2908) and (29.7764,21.6219) .. (30.3075,19.7418) .. controls (30.8386,17.8617) and (30.6503,15.7879) .. (29.7894,14.0341);
		\end{tikzpicture}
		\caption{\label{Figure:EndedVertexZero} The circle $C$ intersects with edges $g_1$ and $g_2$}
	\end{center}
\end{figure}

%% file: FlatKnotoid1.tex
\begin{figure}[h]
	\begin{center}
		\begin{tikzpicture}[scale=0.2, yscale=-1.0]
			\path[draw=black,line cap=round,line width=\figuresLineThickness] (2.2629,10.0584) .. controls (3.0193,10.9302) and (4.0246,11.5836) .. (5.1286,11.9207) .. controls (6.2325,12.2578) and (7.4314,12.2775) .. (8.5458,11.9769) .. controls (9.6602,11.6763) and (10.6865,11.0564) .. (11.4712,10.2099) .. controls (12.2558,9.3634) and (12.7962,8.2930) .. (13.0116,7.1590) .. controls (13.2301,6.0084) and (13.1039,4.7691) .. (12.5187,3.7546) .. controls (12.2260,3.2473) and (11.8219,2.8022) .. (11.3337,2.4788) .. controls (10.8454,2.1554) and (10.2728,1.9560) .. (9.6880,1.9261) .. controls (9.0880,1.8955) and (8.4828,2.0438) .. (7.9546,2.3301) .. controls (7.4265,2.6165) and (6.9754,3.0383) .. (6.6367,3.5345) .. controls (5.9591,4.5267) and (5.7470,5.7911) .. (5.9047,6.9822) .. controls (6.0591,8.1488) and (6.5562,9.2680) .. (7.3181,10.1647) .. controls (8.0800,11.0615) and (9.1041,11.7328) .. (10.2304,12.0737) .. controls (11.3568,12.4145) and (12.5812,12.4237) .. (13.7126,12.0999) .. controls (14.8439,11.7761) and (15.8780,11.1203) .. (16.6534,10.2351);
			\path[fill=black] (2.2629,10.0584) circle (\figuresPointSize);
			\path[fill=black] (16.6534,10.2351) circle (\figuresPointSize);
		\end{tikzpicture}
		\caption{\label{Figure:FlatKnotoid1}The unique FKD with one crossing}
	\end{center}
\end{figure}
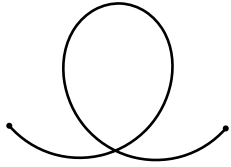

%% file: NonPrimeReduction0102.tex
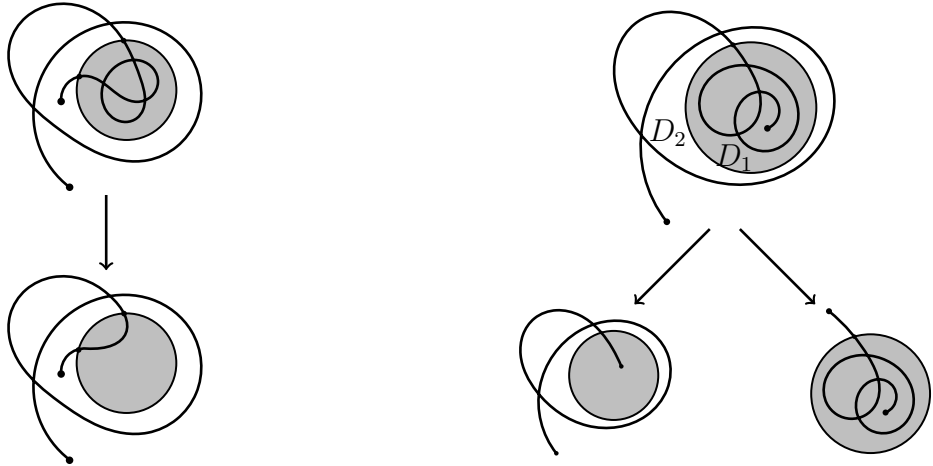
\begin{figure}[h]
	\begin{center}
		\begin{minipage}{0.4\textwidth}
			\begin{center}
				\begin{tikzpicture}[scale=0.25, yscale=-1.0]
					\path[draw=black,fill=lightgray,line width=\figuresThinLineThickness] (7.2406,6.4424) circle (2.65);
					\path[draw=black,line cap=round,line width=\figuresLineThickness] (4.2140,11.6095) .. controls (3.4969,11.0627) and (2.9359,10.3136) .. (2.6128,9.4716) .. controls (2.2898,8.6297) and (2.2057,7.6976) .. (2.3729,6.8115) .. controls (2.5402,5.9253) and (2.9582,5.0879) .. (3.5658,4.4217) .. controls (4.1735,3.7554) and (4.9690,3.2623) .. (5.8360,3.0144) .. controls (6.7233,2.7607) and (7.6885,2.7661) .. (8.5613,3.0658) .. controls (9.4342,3.3654) and (10.2077,3.9639) .. (10.6854,4.7535) .. controls (11.2821,5.7399) and (11.3856,7.0136) .. (10.9372,8.0757) .. controls (10.7129,8.6067) and (10.3583,9.0825) .. (9.9117,9.4470) .. controls (9.4652,9.8115) and (8.9273,10.0638) .. (8.3610,10.1714) .. controls (7.5180,10.3317) and (6.6372,10.1700) .. (5.8441,9.8422) .. controls (5.0511,9.5144) and (4.3348,9.0269) .. (3.6454,8.5159) .. controls (2.8885,7.9549) and (2.1440,7.3484) .. (1.6190,6.5660) .. controls (1.3566,6.1749) and (1.1521,5.7418) .. (1.0454,5.2830) .. controls (0.9386,4.8242) and (0.9317,4.3390) .. (1.0535,3.8839) .. controls (1.1907,3.3713) and (1.4910,2.9058) .. (1.8922,2.5585) .. controls (2.2934,2.2113) and (2.7931,1.9819) .. (3.3157,1.8899) .. controls (3.8383,1.7980) and (4.3825,1.8420) .. (4.8881,2.0032) .. controls (5.3936,2.1643) and (5.8603,2.4412) .. (6.2540,2.7970) .. controls (7.1344,3.5926) and (7.6150,4.7324) .. (7.9931,5.8571) .. controls (8.1116,6.2094) and (8.2238,6.5713) .. (8.2218,6.9430) .. controls (8.2207,7.1288) and (8.1904,7.3156) .. (8.1203,7.4877) .. controls (8.0503,7.6598) and (7.9394,7.8169) .. (7.7925,7.9307) .. controls (7.6272,8.0586) and (7.4193,8.1285) .. (7.2104,8.1340) .. controls (7.0014,8.1395) and (6.7923,8.0818) .. (6.6117,7.9767) .. controls (6.4311,7.8715) and (6.2789,7.7198) .. (6.1664,7.5437) .. controls (6.0538,7.3676) and (5.9805,7.1675) .. (5.9447,6.9616) .. controls (5.8904,6.6494) and (5.9225,6.3231) .. (6.0330,6.0262) .. controls (6.1435,5.7292) and (6.3318,5.4621) .. (6.5718,5.2551) .. controls (6.7687,5.0853) and (7.0014,4.9552) .. (7.2524,4.8871) .. controls (7.5034,4.8191) and (7.7726,4.8141) .. (8.0233,4.8836) .. controls (8.2739,4.9531) and (8.5046,5.0983) .. (8.6674,5.3011) .. controls (8.8301,5.5039) and (8.9224,5.7645) .. (8.9128,6.0244) .. controls (8.9025,6.3056) and (8.7719,6.5793) .. (8.5639,6.7689) .. controls (8.3559,6.9585) and (8.0739,7.0630) .. (7.7925,7.0611) .. controls (7.5360,7.0594) and (7.2857,6.9729) .. (7.0613,6.8489) .. controls (6.8368,6.7249) and (6.6351,6.5640) .. (6.4362,6.4022) .. controls (6.2373,6.2403) and (6.0390,6.0759) .. (5.8194,5.9435) .. controls (5.5997,5.8111) and (5.3558,5.7110) .. (5.1002,5.6899) .. controls (4.7399,5.6602) and (4.3691,5.7975) .. (4.1150,6.0548) .. controls (3.8610,6.3120) and (3.7283,6.6844) .. (3.7625,7.0444);
					\path[fill=black] (4.2140,11.6095) circle (\figuresPointSize);
					\path[fill=black] (3.7625,7.0444) circle (\figuresPointSize);
					\path[fill=black] (7.0833,3.8050) circle (\figuresSmallPointSize);
					\path[fill=black] (4.7300,5.7166) circle (\figuresSmallPointSize);
				\end{tikzpicture}
				\\
				\begin{tikzpicture}
					\path[draw=black, <-, line width=\figuresLineThickness] (0.0, 0.0) -- (0.0, 1.0);
				\end{tikzpicture}
				\\
				\begin{tikzpicture}[scale=0.25, yscale=-1.0]
					\path[draw=black,fill=lightgray,line width=\figuresThinLineThickness] (18.4921,6.4435) circle (2.65);
					\path[fill=black] (15.4654,11.6106) circle (\figuresPointSize);
					\path[fill=black] (15.0173,7.0204) circle (\figuresPointSize);
					\path[fill=black] (18.3612,3.8047) circle (\figuresSmallPointSize);
					\path[fill=black] (15.9458,5.7387) circle (\figuresSmallPointSize);
					\path[draw=black,line cap=round,line width=\figuresLineThickness] (15.4688,11.5855) .. controls (14.7518,11.0387) and (14.1907,10.2895) .. (13.8677,9.4476) .. controls (13.5446,8.6057) and (13.4606,7.6736) .. (13.6278,6.7874) .. controls (13.7950,5.9013) and (14.2130,5.0639) .. (14.8207,4.3976) .. controls (15.4283,3.7314) and (16.2238,3.2382) .. (17.0908,2.9904) .. controls (17.9781,2.7367) and (18.9434,2.7421) .. (19.8162,3.0417) .. controls (20.6890,3.3414) and (21.4625,3.9399) .. (21.9402,4.7294) .. controls (22.5370,5.7159) and (22.6405,6.9896) .. (22.1920,8.0517) .. controls (21.9678,8.5827) and (21.6132,9.0585) .. (21.1666,9.4230) .. controls (20.7200,9.7875) and (20.1822,10.0397) .. (19.6159,10.1474) .. controls (18.7728,10.3076) and (17.8920,10.1460) .. (17.0990,9.8182) .. controls (16.3059,9.4904) and (15.5897,9.0029) .. (14.9003,8.4919) .. controls (14.1434,7.9309) and (13.3989,7.3243) .. (12.8739,6.5420) .. controls (12.6114,6.1509) and (12.4069,5.7178) .. (12.3002,5.2590) .. controls (12.1935,4.8002) and (12.1865,4.3150) .. (12.3083,3.8599) .. controls (12.4456,3.3473) and (12.7458,2.8818) .. (13.1470,2.5345) .. controls (13.5482,2.1872) and (14.0479,1.9579) .. (14.5705,1.8659) .. controls (15.0931,1.7739) and (15.6373,1.8180) .. (16.1429,1.9792) .. controls (16.6485,2.1403) and (17.1152,2.4172) .. (17.5089,2.7730) .. controls (19.2824,4.3038) and (18.6599,5.7982) .. (16.3551,5.6659) .. controls (15.9947,5.6362) and (15.6240,5.7735) .. (15.3699,6.0307) .. controls (15.1158,6.2880) and (14.9831,6.6604) .. (15.0173,7.0204);
				\end{tikzpicture}
			\end{center}
		\end{minipage}
		\hfill
		\begin{minipage}{0.4\textwidth}
			\begin{center}
				\begin{tikzpicture}[scale=0.22, yscale=-1.0]
					\path[draw=black,fill=lightgray,line width=\figuresThinLineThickness] (32.1147,6.7843) circle (3.95);
					\draw (27.1, 8.3) node {$D_2$};
					\draw (31.1, 9.8) node {$D_1$};
					\path[draw=black,line cap=round,line width=\figuresLineThickness] (27.0303,13.6897) .. controls (25.7313,12.0564) and (25.1833,9.8490) .. (25.5676,7.7978) .. controls (25.9519,5.7466) and (27.2619,3.8873) .. (29.0640,2.8350) .. controls (30.4226,2.0418) and (32.0785,1.7057) .. (33.6054,2.0841) .. controls (34.3689,2.2734) and (35.0911,2.6381) .. (35.6810,3.1584) .. controls (36.2710,3.6786) and (36.7259,4.3552) .. (36.9626,5.1053) .. controls (37.2633,6.0579) and (37.2032,7.1121) .. (36.8388,8.0422) .. controls (36.4744,8.9723) and (35.8144,9.7770) .. (34.9998,10.3553) .. controls (33.5105,11.4125) and (31.5345,11.6863) .. (29.7640,11.2378) .. controls (27.9935,10.7894) and (26.4329,9.6626) .. (25.3039,8.2269) .. controls (24.5597,7.2804) and (23.9802,6.1659) .. (23.8651,4.9673) .. controls (23.8075,4.3681) and (23.8687,3.7544) .. (24.0721,3.1878) .. controls (24.2756,2.6212) and (24.6238,2.1032) .. (25.0911,1.7235) .. controls (25.6395,1.2780) and (26.3400,1.0344) .. (27.0456,0.9993) .. controls (27.7513,0.9642) and (28.4604,1.1318) .. (29.0966,1.4391) .. controls (30.3690,2.0537) and (31.3183,3.1930) .. (32.0201,4.4195) .. controls (32.2956,4.9009) and (32.5419,5.4080) .. (32.6469,5.9526) .. controls (32.7520,6.4973) and (32.7045,7.0873) .. (32.4221,7.5647) .. controls (32.2558,7.8459) and (32.0117,8.0804) .. (31.7244,8.2359) .. controls (31.4371,8.3914) and (31.1075,8.4676) .. (30.7811,8.4547) .. controls (30.4547,8.4418) and (30.1323,8.3400) .. (29.8573,8.1636) .. controls (29.5823,7.9873) and (29.3553,7.7370) .. (29.2059,7.4465) .. controls (28.9741,6.9956) and (28.9355,6.4541) .. (29.0775,5.9674) .. controls (29.2195,5.4807) and (29.5361,5.0511) .. (29.9444,4.7506) .. controls (30.3527,4.4501) and (30.8490,4.2767) .. (31.3540,4.2320) .. controls (31.8590,4.1872) and (32.3718,4.2686) .. (32.8478,4.4431) .. controls (33.2941,4.6068) and (33.7118,4.8533) .. (34.0582,5.1788) .. controls (34.4047,5.5042) and (34.6787,5.9095) .. (34.8343,6.3587) .. controls (34.9739,6.7616) and (35.0170,7.1993) .. (34.9451,7.6196) .. controls (34.8733,8.0400) and (34.6852,8.4416) .. (34.3982,8.7570) .. controls (34.1112,9.0724) and (33.7249,9.2988) .. (33.3067,9.3822) .. controls (32.8885,9.4655) and (32.4405,9.4029) .. (32.0674,9.1965) .. controls (31.6812,8.9828) and (31.3830,8.6184) .. (31.2403,8.2007) .. controls (31.0975,7.7831) and (31.1093,7.3162) .. (31.2634,6.9026) .. controls (31.4021,6.5301) and (31.6603,6.1960) .. (32.0076,6.0029) .. controls (32.1813,5.9063) and (32.3756,5.8458) .. (32.5738,5.8319) .. controls (32.7721,5.8180) and (32.9741,5.8513) .. (33.1552,5.9330) .. controls (33.3584,6.0246) and (33.5337,6.1771) .. (33.6526,6.3657) .. controls (33.7714,6.5542) and (33.8334,6.7782) .. (33.8284,7.0010) .. controls (33.8234,7.2238) and (33.7514,7.4448) .. (33.6242,7.6278) .. controls (33.4970,7.8108) and (33.3151,7.9553) .. (33.1079,8.0377);
					\path[fill=black] (27.0303,13.6897) circle (\figuresPointSize);
					\path[fill=black] (33.1079,8.0377) circle (\figuresPointSize);
					\path[fill=black] (31.0303,2.9889) circle (\figuresSmallPointSize);
				\end{tikzpicture}
				\\
				\begin{tikzpicture}
					\path[draw=black, ->, line width=\figuresLineThickness] (-0.2, 0.0) -- (-1.2, -1.0);
					\path[draw=black, ->, line width=\figuresLineThickness] (0.2, 0.0) -- (1.2, -1.0);
				\end{tikzpicture}
				\\
				\begin{tikzpicture}[scale=0.15, yscale=-1.0]
					\path[draw=black,fill=lightgray,line width=\figuresThinLineThickness] (46.7364,6.7443) circle (3.95);
					\path[fill=black] (41.6519,13.6497) circle (\figuresPointSize);
					\path[fill=black] (47.4119,5.9409) circle (\figuresPointSize);
					\path[fill=black] (45.6448,2.9509) circle (\figuresSmallPointSize);
					\path[draw=black,line cap=round,line width=\figuresLineThickness] (41.6519,13.6497) .. controls (40.3529,12.0164) and (39.8048,9.8089) .. (40.1892,7.7577) .. controls (40.5735,5.7065) and (41.8835,3.8472) .. (43.6857,2.7950) .. controls (45.0442,2.0018) and (46.7001,1.6657) .. (48.2271,2.0442) .. controls (48.9905,2.2334) and (49.7127,2.5982) .. (50.3027,3.1184) .. controls (50.8926,3.6386) and (51.3475,4.3152) .. (51.5843,5.0652) .. controls (51.8850,6.0179) and (51.8250,7.0721) .. (51.4606,8.0023) .. controls (51.0962,8.9324) and (50.4361,9.7371) .. (49.6215,10.3152) .. controls (48.1322,11.3722) and (46.1563,11.6453) .. (44.3860,11.1967) .. controls (42.6157,10.7481) and (41.0553,9.6217) .. (39.9256,8.1869) .. controls (39.1805,7.2405) and (38.5997,6.1260) .. (38.4843,4.9270) .. controls (38.4266,4.3276) and (38.4878,3.7137) .. (38.6917,3.1470) .. controls (38.8955,2.5803) and (39.2446,2.0623) .. (39.7128,1.6835) .. controls (40.2614,1.2395) and (40.9618,0.9978) .. (41.6667,0.9642) .. controls (42.3717,0.9306) and (43.0796,1.0993) .. (43.7149,1.4068) .. controls (44.9854,2.0219) and (45.9339,3.1582) .. (46.6418,4.3794) .. controls (46.9330,4.8819) and (47.1905,5.4039) .. (47.4119,5.9409);
				\end{tikzpicture}
				\hfill
				\begin{tikzpicture}[scale=0.2, yscale=-1.0, baseline=-142.0]
					\path[draw=black,fill=lightgray,line width=\figuresThinLineThickness] (44.8681,20.8388) circle (3.95);
					\path[fill=black] (42.1013,15.3524) circle (\figuresPointSize);
					\path[fill=black] (45.8614,22.0922) circle (\figuresPointSize);
					\path[fill=black] (43.7837,17.0434) circle (\figuresSmallPointSize);
					\path[draw=black,line cap=round,line width=\figuresLineThickness] (42.1013,15.3524) .. controls (43.1697,16.2213) and (44.0797,17.2844) .. (44.7735,18.4740) .. controls (45.0533,18.9536) and (45.3007,19.4611) .. (45.4051,20.0064) .. controls (45.5095,20.5517) and (45.4596,21.1422) .. (45.1756,21.6192) .. controls (45.0086,21.8997) and (44.7643,22.1334) .. (44.4770,22.2885) .. controls (44.1898,22.4436) and (43.8605,22.5196) .. (43.5344,22.5068) .. controls (43.2082,22.4940) and (42.8861,22.3926) .. (42.6111,22.2167) .. controls (42.3361,22.0408) and (42.1090,21.7910) .. (41.9594,21.5010) .. controls (41.7269,21.0503) and (41.6878,20.5085) .. (41.8297,20.0216) .. controls (41.9715,19.5347) and (42.2882,19.1048) .. (42.6967,18.8042) .. controls (43.1052,18.5037) and (43.6018,18.3304) .. (44.1070,18.2859) .. controls (44.6122,18.2413) and (45.1251,18.3229) .. (45.6012,18.4976) .. controls (46.0475,18.6614) and (46.4651,18.9079) .. (46.8116,19.2333) .. controls (47.1581,19.5588) and (47.4321,19.9640) .. (47.5877,20.4131) .. controls (47.7273,20.8161) and (47.7705,21.2538) .. (47.6987,21.6741) .. controls (47.6269,22.0945) and (47.4387,22.4962) .. (47.1517,22.8115) .. controls (46.8647,23.1269) and (46.4784,23.3533) .. (46.0602,23.4367) .. controls (45.6420,23.5200) and (45.1940,23.4574) .. (44.8208,23.2510) .. controls (44.4346,23.0373) and (44.1365,22.6729) .. (43.9937,22.2552) .. controls (43.8509,21.8375) and (43.8627,21.3707) .. (44.0168,20.9571) .. controls (44.1555,20.5846) and (44.4137,20.2505) .. (44.7611,20.0574) .. controls (44.9348,19.9608) and (45.1290,19.9003) .. (45.3272,19.8864) .. controls (45.5255,19.8725) and (45.7275,19.9058) .. (45.9087,19.9875) .. controls (46.1119,20.0791) and (46.2872,20.2316) .. (46.4060,20.4201) .. controls (46.5249,20.6087) and (46.5868,20.8326) .. (46.5818,21.0555) .. controls (46.5768,21.2783) and (46.5049,21.4992) .. (46.3777,21.6823) .. controls (46.2504,21.8653) and (46.0685,22.0098) .. (45.8614,22.0922);
				\end{tikzpicture}
			\end{center}
		\end{minipage}
		\caption{\label{Figure:NonPrimeReduction0102}Contracting the shaded disc $D$ to the point to obtain FKD $F'$ (on the left), or contract discs $D_1$ and $D_2$ to obtain two FKD $F_1$ and $F_2$ (on the right)}
	\end{center}
\end{figure}

%% file: Bridge.tex
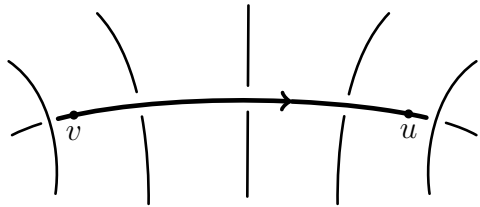
\begin{figure}[h]
	\begin{center}
		\begin{tikzpicture}[scale=0.3, yscale=-1.0]
			\path[draw=black,line cap=round, line width=\figuresLineThickness] (1.2251,8.9269) .. controls (1.5924,8.7379) and (2.0240,8.5673) .. (2.5082,8.4146);
			\path[draw=black,->,line cap=round,line width=\figuresThickLineThickness] (3.2657,8.2006) .. controls (6.0101,7.5056) and (9.9936,7.2765) .. (13.6314,7.4355);
			\path[draw=black,line cap=round,line width=\figuresThickLineThickness] (13.6314,7.4355) .. controls (		15.8699,7.5334) and (17.9774,7.7782) .. (19.5847,8.1520);
			\path[draw=black,line cap=round,line width=\figuresLineThickness] (20.4559,8.3824) .. controls (20.9967,8.5455) and (21.4522,8.7274) .. (21.8014,8.9269);
			\path[draw=black,line cap=round,line width=\figuresLineThickness] (1.0751,5.6517) .. controls (3.0502,7.0018) and (3.3752,9.9520) .. (3.1502,11.5271);
			\path[draw=black,line cap=round,line width=\figuresLineThickness] (21.8014,5.6517) .. controls (19.8262,7.0018) and (19.5012,9.9520) .. (19.7262,11.5271);
			\path[draw=black,line cap=round, line width=\figuresLineThickness] (11.7007,3.1766) -- (11.6797,6.7348);
			\path[draw=black,line cap=round,line width=\figuresLineThickness] (11.6729,7.8927) -- (11.6507,11.6521);
			\path[draw=black,line cap=round,line width=\figuresLineThickness] (5.0003,3.5266) .. controls (5.8650,4.4334) and (6.4094,5.6961) .. (6.7481,7.0144);
			\path[draw=black,line cap=round,line width=\figuresLineThickness] (6.9824,8.1039) .. controls (7.2397,9.5590) and (7.2870,10.9822) .. (7.2755,11.9771);
			\path[draw=black,line cap=round,line width=\figuresLineThickness] (17.9261,3.5266) .. controls (17.0398,4.4562) and (16.4899,5.7596) .. (16.1533,7.1136);
			\path[draw=black,line cap=round,line width=\figuresLineThickness] (15.9465,8.0898) .. controls (15.6870,9.5498) and (15.6394,10.9790) .. (15.6510,11.9771);
			\path[fill=black] (3.9640,8.0392) circle (\figuresPointSize) node[below] {$v$};
			\path[fill=black] (18.7964,7.9863) circle (\figuresPointSize) node[below] {$u$};
		\end{tikzpicture}
		\caption{\label{Figure:Bridge}The bridge $B$ starts at $v$ and ends at $u$}
	\end{center}
\end{figure}